\documentclass[12pt,a4paper,reqno]{amsart}
\usepackage{amsmath}
\usepackage{amsfonts}
\usepackage{amssymb}
\usepackage{amsthm}
\usepackage{enumerate}
\usepackage{centernot}
\usepackage{fullpage}
\usepackage{mathrsfs}
\usepackage{graphicx}

\newcommand{\upperRomannumeral}[1]{\uppercase\expandafter{\romannumeral#1}}

\theoremstyle{plain}

  \newtheorem{proposition}[]{Proposition}
  \newtheorem{lemma}[]{Lemma}
  \newtheorem{theorem}[]{Theorem}
  \newtheorem{corollary}[]{Corollary}
  
  \newtheorem{remark}[]{Remark}
  \newtheorem{definition}[]{Definition}

\title{Exponential decay for the near-critical scaling limit of the planar Ising model}
\author{Federico Camia}
\address{Division of Science, NYU Abu Dhabi, Saadiyat Island, Abu Dhabi, UAE \& Department of Mathematics, VU Amsterdam, De Boelelaan 1081a, 1081 HV Amsterdam, the Netherlands.}
\email{federico.camia@nyu.edu}
\author{Jianping Jiang}
\address{NYU-ECNU Institute of Mathematical Sciences at NYU Shanghai, 3663 Zhongshan
Road North, Shanghai 200062, China.}
\email{jjiang@nyu.edu}
\author{Charles M. Newman}
\address{Courant Institute of Mathematical Sciences, New York University,
251 Mercer st, New York, NY 10012, USA, \& NYU-ECNU Institute of Mathematical
Sciences at NYU Shanghai, 3663 Zhongshan Road North, Shanghai 200062, China.}
\email{newman@cims.nyu.edu}


\begin{document}
\begin{abstract}
We consider the Ising model at its critical temperature with external magnetic field $ha^{15/8}$ on the square lattice with lattice spacing $a$. We show that the truncated two-point function in this model decays exponentially with a rate independent of $a$. As a consequence, we show exponential decay in the near-critical scaling limit Euclidean magnetization field. For the lattice model with $a=1$, the mass (inverse correlation length) is of order $h^{8/15}$ as $h\downarrow 0$; for the Euclidean field, it equals exactly $Ch^{8/15}$ for some $C$. Although there has been much  progress in the study of critical scaling limits, results on near-critical models are far fewer due to the lack of conformal invariance away from the critical point.  Our arguments combine lattice and continuum FK representations, including coupled conformal loop and measure ensembles, showing that such ensembles can be useful even in the study of near-critical scaling limits. Thus we provide the first substantial application of measure ensembles.
\end{abstract}
\maketitle
\section{Introduction}
In this paper we obtain the first proof of exponential decay (or equivalently, a mass gap lower bound) for the important Euclidean field theory that is a near-critical scaling limit of the planar Ising model at the critical temperature, with an external magnetic field. A.B.~Zamolodchikov proposed \cite{Zam89,Zam89bis} a solution, directly in the scaling limit, in terms of scattering amplitudes for eight relativistic particles. Since the Ising model with an external magnetic field has never been solved on a lattice, Zamolodchikov's solution came as a major surprise (see \cite{Del04} for a recent review), and has not yet been put on firm mathematical ground, despite having striking implications for the Ising model and beyond (see \cite{Del04}). In relation to Zamolodchikov's scattering theory, our mass gap result basically shows the existence of at least one particle with strictly positive mass. As a corollary of our main results, we also provide a rigorous proof of the power-law behavior of the correlation length for the planar Ising model at the critical temperature, as the external magnetic field tends to zero. Key to our arguments is the use of conformal measure ensembles, introduced in \cite{CN09}, where they were called cluster area measures, and then constructed for percolation and the FK (Fortuin-Kasteleyn)-Ising model in \cite{CCK17}. The FK representation (see \cite{Gri06})
has been an invaluable tool in studies of the Ising model --- particularly for the critical two-dimensional scaling limit, where it is closely related to conformal loop ensembles \cite{She09,SW12}. Here we extend that approach to the near-critical case by means of a coupling between FK and Ising variables in the presence of an external field and by coupled measure and loop ensembles. An upper bound for the mass gap is obtained using methods quite different from those of the rest of the paper, namely transfer matrix techniques and reflection positivity. An upper bound using similar methods to the lower bound methods of this paper is in \cite{CJN19}.

\subsection{Overview}\label{secoverview}
The Ising model \cite{Isi25}, suggested by Lenz \cite{Len20} and cast in its current form by Peierls \cite{Pei36}, is one of the most studied models of statistical mechanics. Its two-dimensional
version has played a special role since Peierls'
proof of a phase transition \cite{Pei36}, and Onsager's calculation of the free
energy \cite{Ons44}. This phase transition has become a prototype for developing new techniques.
Its analysis has helped test a fundamental tenet of critical phenomena,
that near-critical physical systems are characterized by a \emph{correlation length}, which provides the natural length scale for the system, and diverges when the critical point is approached.

This divergence implies that the critical system itself has no characteristic length and
is therefore invariant under scale transformations. This in turn suggests that thermodynamic
functions at criticality are homogeneous, and predicts the appearance
of power laws. For a lattice-based model, it also means that, at or near criticality, it
should be possible to rescale the model appropriately and obtain a continuum scaling
limit by sending the lattice spacing to zero. This idea is at the heart of the renormalization group philosophy.

Thanks to the work of Polyakov \cite{Pol70} and others \cite{BPZ84a,BPZ84b}, it was understood that, once an appropriate continuum scaling limit is taken,
critical models should acquire conformal invariance. Because the conformal group is in general a
finite dimensional Lie group, the resulting constraints are limited in number; however, in two dimensions, since every analytic
function $f$ defines a conformal transformation, provided that $f^{\prime}$ is nonvanishing, the conformal group is infinite-dimensional.

Following this observation, in two dimensions, conformal methods were applied extensively to Ising and
Potts models, Brownian motion, the self-avoiding walk, percolation, and diffusion limited
aggregation. The large body of knowledge and techniques that resulted goes under the
name of Conformal Field Theory (CFT). The aspect of CFT most related to our work in this paper is a particular near-critical scaling limit of the two-dimensional Ising model believed to be related to the Lie algebra $E_8$ \cite{Zam89,Del04,BG11,MM12}, which we discuss in more detail below.

In recent years, significant developments in two-dimensional
critical phenomena have emerged in the mathematics literature. A major breakthrough was the introduction by Schramm \cite{Sch00} of the Schramm-Loewner Evolution (SLE) and its subsequent analysis and application to the
scaling limit problem for several models, most notably by Lawler, Schramm and Werner
\cite{LSW04}, and by Smirnov \cite{Smi01} (see also \cite{CN07}). The subsequent introduction of Conformal
Loop Ensembles (CLEs) \cite{CN04,CN06,She09,SW12,Wer03}, which are collections of SLE-type closed
curves, provided an additional tool to analyze the scaling limit geometry of critical models.
Substantial progress in the rigorous analysis of the two-dimensional Ising model at criticality
was made by Smirnov \cite{Smi10} with the introduction and scaling limit analysis of
\emph{fermionic observables}, also known as \emph{discrete holomorphic observables} or \emph{holomorphic fermions}.
These have proved extremely useful in studying the Ising model in finite geometries with boundary conditions and in establishing
conformal invariance of the scaling limit of various quantities, including the energy density
\cite{Hon10,HS13} and spin correlation functions \cite{CHI15}. (An independent derivation of critical
Ising correlation functions in the plane was obtained in \cite{Dub11}.)

In \cite{CGN15} (resp., \cite{CGN16}), it was shown that the critical Ising model (resp., near-critical model with external magnetic field $ha^{15/8}$) on the rescaled lattice $a\mathbb{Z}^2$ has a scaling limit $\Phi^0$ (resp., $\Phi^{h}$) as $a\downarrow 0$ --- denoted then by $\Phi^{\infty}$ (resp., $\Phi^{\infty, h}$). $\Phi^0$ satisfies the expected conformal covariance properties \cite{CGN15}. When $h\neq 0$, it was also expected (as stated in \cite{CGN15}) that the truncated correlations of the near-critical scaling limit would decay exponentially. In this paper, we give a proof of that statement and we rigorously verify that the critical exponent for how the correlation length diverges as $h\downarrow 0$ is $8/15$, together with the related scaling properties of $\Phi^h$.

$\Phi^h$ is a (generalized) random field on $\mathbb{R}^2$ --- i.e., for suitable test functions $f$ on $\mathbb{R}^2$, there are random variables $\Phi^h(f)$, formally written as $\int_{\mathbb{R}^2}\Phi^h(x)f(x)dx$. Euclidean random fields such as $\Phi^h$ on the Euclidean ``space-time'' $\mathbb{R}^d:=\{x=(x_0,w_1,\ldots,w_{d-1})\}$ (in our case $d=2$) are related to quantum fields on relativistic space time, $\{(t,w_1,\ldots,w_{d-1})\}$, essentially by replacing $x_0$ with a complex variable and analytically continuing from the purely real $x_0$ to a pure imaginary $(-it)$ --- see \cite{OS73}, Chapter 3 of \cite{GJ87} and \cite{MM97} for background. One major reason for interest in $\Phi^h$ is that the associated quantum field is predicted \cite{Zam89,Zam89bis} to have remarkable properties including relations between the masses of particles described by the quantum field and the Lie algebra $E_8$ --- see \cite{Del04,BG11,MM12}. A natural first step in analyzing particle masses is to prove a strictly positive lower bound $m(h)$ on all masses (i.e., a mass gap) which exactly corresponds (see \cite{Sim73,Spe74} and Chapters \upperRomannumeral{7} and \upperRomannumeral{11} of \cite{GJ85}) to the type of exponential decay we prove in this paper --- i.e., showing (as a consequence of Theorem \ref{thmmain} below) that for test functions $f,g\geq 0$ of compact support, and some $C=C(f,g)<\infty$,
\[0\leq\text{Cov}\left(\Phi^h(f),\Phi^h(T^ug)\right)\leq C(f,g)e^{-m(h)u} \text{ for } u\geq 0,\]
where $(T^ug)(x_0,w_1)=g(x_0-u,w_1)$.

$\Phi^h$ is the limit, as the lattice spacing $a\downarrow 0$, of the lattice field
\begin{equation}\label{eqPhi}
\Phi^{a,h}:=a^{15/8}\sum_{x\in a\mathbb{Z}^2}\sigma_x\delta_x,
\end{equation}
where $\{\sigma_x\}_{x\in a\mathbb{Z}^2}$ are the $\pm1$-valued spin variables in the standard planar Ising model (on $a\mathbb{Z}^2$) at the critical (inverse) temperature $\beta=\beta_c$ with magnetic field $H=a^{15/8}h$ and $\delta_x$ is a unit Dirac point measure at $x$. Hence, obtaining an exponential decay result for $\Phi^{a,h}$ is directly related to corresponding results for $\{\sigma_x\}$ on the lattice, which we discuss next. But first we note that the choice of scaling factor $a^{15/8}$ in \eqref{eqPhi} relies on Wu's celebrated result (see \cite{Wu66} and \cite{MW73}) that the critical Ising two-point function decays precisely as $C^{\prime}|x-y|^{-1/4}$ for some $C^{\prime}$ (where $|x-y|:=\|x-y\|_2$, the Euclidean distance).

It was first proved in \cite{LP68} that the lattice truncated two-point function with $H>0$ decays exponentially. See also \cite{FR15} for a different and simpler proof, where it was also shown that the decay rate $\tilde{m}(H)$ (or inverse correlation length) on $\mathbb{Z}^2$ is bounded below linearly in $H$. In this paper, we show exponential decay for the near-critical Ising model on $a\mathbb{Z}^2$ with $H=a^{15/8}h$. Roughly speaking, this means (see Theorem \ref{thm1} below) that there is a lower bound on $\tilde{m}(H)$ behaving like $H^{8/15}$ as $H\downarrow0$.

Good lower bounds as $a\downarrow 0$ for fixed $h$ or as $H\downarrow 0$ for fixed $a$ seem essential in order to obtain an exponential decay rate for the continuum field $\Phi^h$ for any particular value, say $h_0$, of the renormalized field strength $h$. It is worth noting that in the earlier work of \cite{LP68,FR15} on lattices, exponential decay was first obtained for large $H$ (by expansion techniques) and then shown to apply to all $H>0$, albeit with a sub-optimal lower bound on $\tilde{m}(H)$ as $H\downarrow 0$. However, in the continuum setting, exponential decay (i.e., $m(h)>0$) for any single value $h_0\neq 0$ of $h$ immediately implies exponential decay for all $h\neq 0$ with the correct dependence of $m(h)$ on $h$. This follows from simple scaling properties of $\Phi^h$ as we now explain.

Both the $h=0$ and $h>0$ fields $\Phi^0$ and $\Phi^h$ can be defined on a bounded (simply-connected) domain in $\mathbb{R}^2$ (now thought of as the complex plane $\mathbb{C}$) with appropriate boundary condition (e.g., free or plus) as well as on the full plane. Conformal mapping properties for $\Phi^0$ were given in Theorem 1.8 of \cite{CGN15}. Similar properties for $\Phi^h$ are only implicit in \cite{CGN16} so we state them explicitly below as Theorem \ref{thmcon} in Section \ref{secscl}. In the case of the full plane one can consider (for $h=0$ and $h>0$) the conformal mapping, $x\rightarrow \lambda x$, with $\lambda>0$, by defining $\Phi^h_{\lambda}(x)=\Phi^h(\lambda x)$, by which we mean $\Phi^h_{\lambda}(f)=\lambda^{-2}\Phi^h(f_{\lambda^{-1}})$ with $f_{\lambda^{-1}}(x)=f(\lambda^{-1}x)$. Indeed, the fields $\Phi^h$ are not defined pointwise, but it is sometimes convenient to treat them, with an abuse of notation, as if they were. By doing so, one can write that $\lambda^{1/8}\Phi^{h_0}(\lambda x)$ is equal in distribution to $\Phi^{\lambda^{15/8}h_0}(x)$ for any $\lambda>0$ and real $h_0$. Thus a positive exponential decay rate $m(h_0)>0$, for a single $h_0>0$, implies the same for all $h\neq 0$ with $m(h)=\left(m(h_0)/h_0^{8/15}\right)h^{8/15}$.

Exponential upper bounds of the form $Ce^{-m(h)|x-y|}$ for the truncated two-point function $\langle \sigma_x;\sigma_y\rangle_{a,h}:=\text{Cov}_{a,h}(\sigma_x,\sigma_y)$ on $a\mathbb{Z}^2$ for small $a$ or for the corresponding continuum $G^h(x-y):=E\left(\Phi^{h}(x)\Phi^{h}(y)\right)-E\left(\Phi^{h}(x)\right)E\left(\Phi^{h}(y)\right)$ on $\mathbb{R}^2$ (where $E\left(\Phi^{h}(x)\Phi^{h}(y)\right)-E\left(\Phi^{h}(x)\right)E\left(\Phi^{h}(y)\right)$ may be obtained as the scaling limit of the corresponding quantity on the lattice) cannot be valid for small $|x-y|$ since when $h=0$, $G^0(x-y)=C|x-y|^{-1/4}$. Indeed, one expects exponential decay only for $|x-y|$ larger than the correlation length and otherwise $G^0(x-y)$ behavior. Since the GHS inequality \cite{GHS70} implies $G^0(x-y)\geq G^h(x-y)$ for all $x,y$, one can paste together exponential upper bounds for large $|x-y|$ with the $h=0$ upper bounds for small $|x-y|$ to obtain an upper bound of the form $C^{\prime}|x-y|^{-1/4}e^{-m^{\prime}(h)|x-y|}$ for all $|x-y|$, as we do in Theorems \ref{thm1} and \ref{thmmain} below.

The analysis of Theorem \ref{thm1} is done in Section \ref{seclargeh} after reviewing in Section \ref{secpre} the FK random cluster representation for the Ising model and discussing couplings of FK and Ising variables relevant when $h>0$.  The heart of that analysis consists of the first five lemmas in that section, which concern circuits of vertices in an annulus created by ``necklaces'' of touching FK-open clusters containing sufficiently many vertices. For large $h$, with high probability, a necklace and its circuit will have all $+1$ spin values; this will also be true for small $h$ by changing the scale of the boxes used in the argument. Correlations will then only occur between regions of $a\mathbb{Z}^2$ that are connected within the complement of a strongly supercritical infinite percolation cluster. The proof relies on continuum results concerning coupled conformal loop and measure ensembles, denoted CLE and CME respectively. Indeed, a main contribution of this paper is a demonstration of the utility of such coupled loop and measure ensembles. Relevant CLE$_{\kappa}$ results are in \cite{She09}, \cite{SW12}, \cite{MSW16}. CME$_{\kappa}$ and its coupling to CLE$_{\kappa}$ was proposed in \cite{CN09} and carried out in \cite{CCK17} for $\kappa=6$ and $16/3$. It may be worth noting, as was mentioned in \cite{CN09}, that, in addition to their utility for near-critical models, measure ensembles may be more extendable than loop ensembles to scaling limits in dimensions $d>2$, but that issue goes well beyond
the scope of this paper. In Section \ref{seccon}  the continuum field $\Phi^h$ is studied, including conformal mapping properties. In Appendix \upperRomannumeral{1}, we state some of the key ingredients used as building blocks for our results.

In Appendix \upperRomannumeral{2}, we give a proof of Theorem \ref{thmupper} using reflection-positivity methods. This provides an upper bound for the mass gap (the inverse correlation length) matching the lower bound of Corollary \ref{cor1}. At the time the first version of this paper was written and posted (July, 2017), there was no proof of an upper bound based on FK methods; that changed a bit later with the proof presented in \cite{CJN19}. Accordingly, we have now placed the original, and  much shorter, proof based on refection positivity in Appendix \upperRomannumeral{2}.

\subsection{Main results}

Let $a>0$. Denote by $P_h^a$ the infinite volume Ising measure at the inverse critical temperature $\beta_c$ on $a\mathbb{Z}^2$ with external field $a^{15/8}h>0$. The precise value of $\beta_c$, $\log(1+\sqrt{2})/2$, originates in \cite{KW41,Ons44}. Let $\langle\cdot\rangle_{a,h}$ be the expectation with respect to $P_h^a$. Let $\langle \sigma_x;\sigma_y\rangle_{a,h}$ be the truncated two-point function, i.e.,
\[\langle \sigma_x;\sigma_y\rangle_{a,h}:=\langle \sigma_x\sigma_y\rangle_{a,h}-\langle\sigma_x\rangle_{a,h}\langle\sigma_y\rangle_{a,h}.\]
Our main result about the truncated two-point function is:
\begin{theorem}\label{thm1}
There exist $B_0, C_0\in(0,\infty)$ such that for any $a\in(0,1]$ and $h\in(0, a^{-15/8}]$
\begin{equation}\label{eqdis1}
0\leq \langle \sigma_x;\sigma_y\rangle_{a,h}\leq C_0a^{1/4}|x-y|^{-1/4}e^{-B_0h^{8/15}|x-y|}\text{ for any }x,y\in a\mathbb{Z}^2.
\end{equation}
 In particular, for a=1 and any $H\in(0,1]$, we have
\begin{equation}\label{eqdis2}
0\leq \langle \sigma_{x^{\prime}};\sigma_{y^{\prime}}\rangle_{1,H}\leq C_0 |x^{\prime}-y^{\prime}|^{-1/4} e^{-B_0H^{8/15}|x^{\prime}-y^{\prime}|} \text{ for any }x^{\prime},y^{\prime}\in \mathbb{Z}^2.
\end{equation}
\end{theorem}

\begin{remark}
By the GHS inequality \cite{GHS70}, $\langle \sigma_x;\sigma_y\rangle_{a,h}$ is decreasing in $h$ for fixed $a, x, y$. Thus  \eqref{eqdis1} implies that for any $a\in(0,1]$ and $h>a^{-15/8}$
\begin{equation*}
0\leq \langle \sigma_x;\sigma_y\rangle_{a,h}\leq C_0a^{1/4}|x-y|^{-1/4}e^{-B_0a^{-1}|x-y|}\text{ for any }x,y\in a\mathbb{Z}^2.
\end{equation*}
\end{remark}

For $a=1$, define the (lattice) mass (or inverse correlation length) $\tilde{M}(H)$ as the supremum of all $\tilde{m}>0$ such that for some $C_{\tilde{m}}<\infty$,
\begin{equation}\label{equpper1}
\langle \sigma_{x^{\prime}};\sigma_{y^{\prime}}\rangle_{1,H}\leq C_{\tilde{m}}e^{-\tilde{m}|x^{\prime}-y^{\prime}|}\text{ for any }x^{\prime},y^{\prime}\in \mathbb{Z}^2.
\end{equation}
The following immediate corollary of Theorem \ref{thm1} gives a one-sided bound for the behavior of $\tilde{M}(H)$ as $H\downarrow 0$, with the expected critical exponent $8/15$.
\begin{corollary}\label{cor1}
\[\tilde{M}(H)\geq B_0 H^{8/15} \text{ as }H\downarrow0.\]
\end{corollary}


Let $\Phi^{a,h}$ be the near-critical magnetization field in the plane defined by
\begin{equation}\label{eqfield}
\Phi^{a,h}:=a^{15/8}\sum_{x\in a\mathbb{Z}^2}\sigma_x\delta_x,
\end{equation}
where $\{\sigma_x\}_{x\in a\mathbb{Z}^2}$ is a configuration for the measure $P_h^a$. In Theorem 1.4 of \cite{CGN16}, it was proved that $\Phi^{a,h}$ converges in law to a continuum (generalized) random field $\Phi^{h}$. Let $C_0^{\infty}(\mathbb{R}^2)$ denote the set of infinitely differentiable functions with compact support. $\Phi^{h}(f)$ denotes the field $\Phi^{h}$ paired against the test function $f$ (which was denoted $\langle\Phi^{h},f\rangle$ in \cite{CGN16}).
\begin{theorem}\label{thmmain}
For any $f, g\in C_0^{\infty}(\mathbb{R}^2)$, we have
\begin{align*}
\left|~\emph{Cov}\left(\Phi^{h}(f),\Phi^{h}(g)\right) \right|\leq C_0\int\int_{\mathbb{R}^2 \times \mathbb{R}^2}|f(x)||g(y)||x-y|^{-1/4} e^{-B_0h^{8/15}|x-y|}dxdy,
\end{align*}
where $C_0$ and $B_0$ are as in Theorem \ref{thm1}.
\end{theorem}
\begin{remark}
Theorem \ref{thmmain} may be expressed as
\[E\left(\Phi^{h}(x)\Phi^{h}(y)\right)-E\left(\Phi^{h}(x)\right)E\left(\Phi^{h}(y)\right)\leq C_0|x-y|^{-1/4} e^{-B_0h^{8/15}|x-y|}.\]
\end{remark}
For $\Phi^h$, define the mass $M(\Phi^h)$ as the supremum of all $\tilde{m}>0$ such that for all $f,g\in C_0^{\infty}(\mathbb{R}^2)$ and some $C_{\tilde{m}}(f,g)<\infty$,
\begin{equation}\label{equpper2}
\left|~\text{Cov}\left(\Phi^h(f),\Phi^h(T^ug)\right) \right|\leq C_{\tilde{m}}(f,g) e^{-\tilde{m}u} \text{ for } u\geq 0.
\end{equation}
The following corollary is essentially a consequence of Theorem \ref{thmmain} and the scaling properties of $\Phi^h$; to show that $C<\infty$, we use \eqref{equpper3.94} from Appendix \upperRomannumeral{2}. The scaling properties were discussed in Subsection \ref{secoverview} and are presented with more detail in Subsection \ref{secscl}; the proof of Corollary \ref{cor:mass}, including that $C<\infty$, is given in Subsection \ref{seccor2}.
\begin{corollary} \label{cor:mass}
\[M(\Phi^h)=C h^{8/15} \text{ for some } C\in (0,\infty) \text{ and all }h.\]
\end{corollary}
\begin{remark}
Theorem \ref{thmmain} implies (see the remarks after Theorem 2.1 of \cite{Spe74} and Theorem 6 of \cite{Sim73} as well as Chapters \upperRomannumeral{7} and \upperRomannumeral{11} of \cite{GJ85}) the existence of a mass gap in the spectrum of the Hamiltonian of the quantum field theory determined by the Euclidean field $\Phi^{h}$.
\end{remark}

Our final theorem gives a complementary bound to Corollary \ref{cor1} --- i.e., $\tilde{M}(H)\leq C H^{8/15}$ as $H\downarrow 0$. The proof is given in Appendix \upperRomannumeral{2} and is based on reflection positivity. A different proof using FK methods can be found in \cite{CJN19}.
\begin{theorem}\label{thmupper}
\[\limsup_{H\downarrow 0} \tilde{M}(H)/H^{8/15}\leq C \in (0,\infty)\]
with $C$ the same constant as in Corollary \ref{cor:mass}.
\end{theorem}
\begin{remark}
Corollary \ref{cor1} and Theorem \ref{thmupper} combine to give
\[B_0H^{8/15}\leq \tilde{M}(H)\leq (C+\epsilon) H^{8/15}\]
for any $\epsilon >0$ and small $H>0$, with $B_0, C\in (0,\infty)$. This is a strong version of showing that the ($H\downarrow 0$ at $\beta_c$) Ising correlation critical exponent is $8/15$:
\[\lim_{H\downarrow 0}\log(\tilde{M}(H))/\log(H)=8/15.\]
This result complements that of \cite{CJN19} (which improved the result of \cite{CGN14}) that the ($H\downarrow 0$ at $\beta_c$) Ising magnetization exponent is $1/15$:
\[\lim_{H\downarrow 0}\frac{\langle \sigma_0\rangle_{1,H}}{H^{1/15}}=B\in(0,\infty).\]
\end{remark}

\subsection{Description of the proof of exponential decay}
For the reader's convenience, in this subsection we sketch the main arguments of the proof of exponential decay of the truncated two-point function, which represents the core of the paper.
We assume that the reader has some familiarity with FK percolation and the Edwards-Sokal coupling of the Ising model with FK percolation --- including when there is an external magnetic field implemented by couplings to a ghost vertex, denoted $g$. This coupling is discussed with more details in Section \ref{secpre} below. With this knowledge, the notation in this subsection should be self-explanatory (e.g., we use
$\{x\longleftrightarrow y\centernot\longleftrightarrow g\}$ to denote the event that vertices $x$ and $y$ are in the same FK-open cluster and that that cluster is not connected to the ghost). Precise definitions are given later on, when we present the actual proofs.

In our arguments, we use in a crucial way a version of the Edwards-Sokal coupling which makes reference to whole FK clusters rather than individual vertices --- i.e., we consider the clusters formed by the open edges within $\mathbb{Z}^2$ and whether those whole clusters are connected to the ghost. This approach, which is discussed in greater detail in \cite{CJN19}, allows us to express the Radon-Nikodym derivative of the distribution of the FK-open clusters in FK percolation with a ghost vertex with respect to the distribution of the clusters in the model without a ghost vertex in terms of the areas of the FK-open clusters (see \eqref{eqRN}). This coupling also allows us to write the probability of each FK-open cluster to be plus in terms of the size of the cluster (see Proposition~\ref{propFKcon}), a fact that we'll exploit in the proof.

The first step of the proof of exponential decay consists in writing (see Lemma~\ref{lemES})
\begin{eqnarray}
\langle\sigma_x;\sigma_y\rangle_{a,h} & = & \mathbb{P}_h^a(x\longleftrightarrow y)-\mathbb{P}_h^a(x\longleftrightarrow g)\mathbb{P}_h^a(y\longleftrightarrow g) \nonumber \\
& = & \mathbb{P}_h^a(x\longleftrightarrow y \centernot\longleftrightarrow g) - \text{Cov}_h^a(1_{\{x\longleftrightarrow g\}}, 1_{\{y\longleftrightarrow g\}}), \label{eqdifference}
\end{eqnarray}
where $\text{Cov}_h^a$ is the covariance of the FK measure $\mathbb{P}_h^a$ on $a\mathbb{Z}^2$ corresponding to the Ising model with external field $H=ha^{15/8}$ and $1_{\{\cdot\}}$ is the indicator function. It may be worth noting that in the first 3 versions of this paper on arXiv, we proved exponential decay first for large $h$ and then extended it to general $h>0$ by conformal covariance. In this version, we combine the arguments for large $h$ and small $h$, as suggested by a referee. There are many ways to prove exponential decay for large $h$ (or more accurately for large $H$ and fixed $a$); a key feature of this paper is that we obtain the correct dependence of the correlation length on $H$ as $H\downarrow 0$.

Letting $B(x,L)$ denote the square centered at $x$ of side length $2L$ and writing
\begin{align*}
A^c_x:=\{&\text{there exists an FK-open path from } x, \text{ within }B(x,|x-y|/3), \text { to some } \\
&w \text{ with the edge from } w \text{ to } g \text{ open}  \}
\end{align*}
and $A^f_x:=\{x\longleftrightarrow g\}\setminus A^c_x$, so that $\{x\longleftrightarrow g\} = A^c_x \cup A^f_x$, the covariance in \eqref{eqdifference}
can be written as a sum of four covariances and $\langle\sigma_x;\sigma_y\rangle_{a,h}$ as a sum of five terms. Bounding four of these five terms reduces to showing that
\begin{equation} \label{eqstar}
\mathbb{P}_h^a(g \centernot\longleftrightarrow x \longleftrightarrow \partial B(x,|x-y|/3)) \leq \tilde C(h) a^{1/8} e^{\hat C(h)|x-y|}.
\end{equation}
The remaining term, $\text{Cov}_h^a(1_{A_x^c}, 1_{A_y^c})$, needs a separate argument and will be discussed later.

Focusing for now on \eqref{eqstar}, the power law part of the upper bound comes from a 1-arm argument (see Lemma~\ref{lem1arm}), while the exponential part requires a more sophisticated argument that makes use of the conformal measure ensemble, CME$_{16/3}$, coupled to CLE$_{16/3}$~\cite{CCK17} as well as a stochastic domination theorem by Liggett, Schonmann and Stacey \cite{LSS97}. CLE$_{16/3}$ is the scaling limit of the collection of lattice boundaries of critical FK percolation clusters, which suggests that in the scaling limit, the continuum cluster measures of CME$_{16/3}$ are functions of and hence coupled to the loops of CLE$_{16/3}$; this is indeed the case. Roughly speaking, what we use of the coupled CLE$_{16/3}$ and CME$_{16/3}$ is the fact that, for $K$ large, a realization inside any rectangle is likely to contain a chain of not more than $K$ touching loops that cross the rectangle in the long direction, with the first loop touching one of the short sides of the rectangle and the last loop touching the opposite side (see Fig.~\ref{fig:E(K)}). Moreover, the ``areas'' of the continuum clusters associated to the loops in the chain are likely to be larger than $1/K$, with the probability of the event just described going to one as $K \to \infty$. Back on the lattice, this implies that, inside an appropriate rectangle, one can find with high probability a chain of FK-open clusters one lattice spacing away from each other and crossing the rectangle. Moreover, such clusters will, with high probability, have sizes of order $a^{-15/8}$. Lemma~\ref{lemtanh} ensures that, in the FK model corresponding to the Ising model with external field $ah^{15/8}$, FK-open clusters whose size is of order $a^{-15/8}$ are connected to $g$ with high probability.

Combining all of the above, with the help of the FKG inequality, one can show that, with high probability, a large annulus contains a circuit of FK-open clusters one lattice spacing away from each other, each connected to $g$, such that the circuit disconnects the inner square of the annulus from the outer one (see Fig.~\ref{fig:G}). We call such an annulus \emph{good}.

In order to complete the proof of \eqref{eqstar}, we cover the plane with large overlapping annuli in such a way that their inner squares tile the plane. For each such annulus, the event that it is good happens with high probability. We would like to conclude that good annuli percolate, but the annuli are overlapping, so the events are not independent. To deal with this, one can use a stochastic domination result due to Liggett, Schonmann and Stacey \cite{LSS97}. Now, percolation of good annuli implies that the probability that $x$ is surrounded by a circuit of good annuli contained in a square $B(x,L)$ of size $2L$ centered at $x$ is close to one, exponentially in $L$. But because of planarity, if $x$ is surrounded by a circuit of good annuli contained in $B(x,L)$, the event $\{ g \centernot\longleftrightarrow x \longleftrightarrow \partial B(x,L) \}$ cannot happen. This provides the desired exponential bound.

The remaining term can be written as
\begin{eqnarray*}
\text{Cov}_h^a(1_{A_x^c}, 1_{A_y^c}) & = &
\mathbb{P}_h^a(A^c_x \cap A^c_y) - \mathbb{P}_h^a(A^c_x)\mathbb{P}_h^a(A^c_y) \\
& = & \mathbb{P}_h^a(A^c_y) \big[ \mathbb{P}_h^a(A^c_x | A^c_y) - \mathbb{P}_h^a(A^c_y) \big].
\end{eqnarray*}
A 1-arm argument (see Lemma~\ref{lem1arm}) provides a polynomial upper bound of order $a^{1/8}$ for $\mathbb{P}_h^a(A^c_y)$.
The first step in dealing with the remaining factor consists in showing that $\mathbb{P}_h^a(A^c_x | A^c_y)$ is smaller than the probability of the event $A^c_x$ with wired boundary condition on $B(x,2|x-y|/3)$. A key ingredient in proving this fact is Lemma~\ref{lemstochdom}, whose proof is based on showing the monotonicity of the Radon-Nikodym derivative of a suitable conditional FK measure in $B(x,2|x-y|/3)$ with respect to the FK measure in $B(x,2|x-y|/3)$ with wired boundary condition. The remaining step consists in showing that the probability of $A^c_x$ is not affected much by the boundary condition in $B(x,2|x-y|/3)$, which follows from Proposition~\ref{propFK}.

\section{Preliminary definitions and results}\label{secpre}
\subsection{Ising model and FK percolation}
In this subsection, our definitions and terminology (especially after the ghost vertex is introduced below) follow those of \cite{Ale98}. With vertex set $a\mathbb{Z}^2$, we write $a\mathbb{E}^2$ for the set of nearest neighbour edges of $a\mathbb{Z}^2$. For any $D\subseteq \mathbb{R}^2$, let $D^a:=a\mathbb{Z}^2\cap D$ be the set of points of $a\mathbb{Z}^2$ in $D$, and call it the \textit{$a$-approximation} of $D$. For $\Lambda\subseteq a\mathbb{Z}^2$, define $\Lambda^C:=a\mathbb{Z}^2\setminus\Lambda$,
\[\partial_{in}\Lambda:=\{z\in a\mathbb{Z}^2: z\in \Lambda, z \text{ has a nearest neighbor in }\Lambda^C \}, \]
\[\partial_{ex}\Lambda:=\{z\in a\mathbb{Z}^2: z\notin \Lambda, z \text{ has a nearest neighbor in }\Lambda \}, \]
\[\overline{\Lambda}:=\Lambda\cup \partial_{ex}\Lambda.\]
Let $\mathscr{B}(\Lambda)$ be the set of all edges $\{z,w\}\in a\mathbb{E}^2$ with $z,w\in \Lambda$, and $\overline{\mathscr{B}}(\Lambda)$ be the set of all edges $\{z,w\}$ with $z$ or $w\in \Lambda$. We will consider the extended graph $G=(V,E)$ where $V=a\mathbb{Z}^2\cup\{g\}$ ($g$ is usually called the \textit{ghost vertex} \cite{Gri67}) and $E$ is the set of nearest-neighbor edges of $a\mathbb{E}^2$ plus $\{\{z,g\}:z\in a\mathbb{Z}^2\}$. The edges of $a\mathbb{E}^2$ are called \textit{internal edges} while $\{\{z,g\}:z\in a\mathbb{Z}^2\}$ are called \textit{external edges}. Let $\mathscr{E}(\Lambda)$ be the set of all external edges with an endpoint in $\Lambda$, i.e.,
\[\mathscr{E}(\Lambda):=\left\{\left\{z,g\right\}:z\in \Lambda\right\}.\]

Let $\Lambda_L:=[-L,L]^2$ and $\Lambda^a_{L}$ be its $a$-approximation. The classical Ising model at inverse (critical) temperature $\beta_c$ on $\Lambda^a_{L}$ with boundary condition $\eta\in\{-1,+1\}^{\partial_{ex}\Lambda_L^a}$ and external field $a^{\frac{15}{8}}h\geq 0$ is the probability measure $P_{\Lambda_L,\eta, h}^{a}$ on $\{-1,+1\}^{\Lambda^a_{L}}$  such that for any $\sigma\in \{-1,+1\}^{\Lambda^a_{L}}$,
\begin{equation}\label{eqIsingdef}
P_{\Lambda_L,\eta, h}^{a}(\sigma)=\frac{1}{Z^a_{L,\eta,h}}e^{\beta_c\sum_{\{u,v\}}\sigma_u\sigma_v+\beta_c\sum_{\{u,v\}:u\in\Lambda_L^a, v\in\partial_{ex} \Lambda_L^a}\sigma_u\eta_v+a^{15/8}h\sum_{u\in\Lambda^a_{L}}\sigma_u},
\end{equation}
where the first sum is over all nearest neighbor pairs (i.e., $|u-v|=a$) in $\Lambda^a_{L}$, and $Z^a_{L,\eta,h}$ is the partition function (which is the normalization constant needed to make this a probability measure). $P^a_{\Lambda_L,f, h}$ denotes the probability measure with free boundary condition --- i.e., where we omit the second of the three terms in the exponent of \eqref{eqIsingdef}.

It is known that $P_{\Lambda_L,\eta, h}^{a}$ has a unique infinite volume limit as $L\rightarrow\infty$, which we denote by $P_h^a$. Note that this limiting measure does not depend on the choice of boundary condition (see, e.g., Theorem 1 of \cite{Leb72} or the theorem in the appendix of \cite{Rue72}).

The FK (Fortuin and Kasteleyn) percolation model at $\beta_c$ on $\Lambda^a_{L}$ with boundary condition $\rho\in\{0,1\}^{\overline{\mathscr{B}}\left((\Lambda^a_L)^C\right)\cup\mathscr{E}\left((\Lambda^a_L)^C\right)}$ and with external field $a^{\frac{15}{8}}h\geq0$ is the probability measure $\mathbb{P}_{\Lambda_L,\rho,h}^a$ on $\{0,1\}^{\mathscr{B}(\Lambda^a_L)\cup\mathscr{E}(\Lambda^a_L)}$ such that for any $\omega\in\{0,1\}^{\mathscr{B}(\Lambda^a_L)\cup\mathscr{E}(\Lambda^a_L)}$,
\begin{align}\label{eqFKdef}
\mathbb{P}_{\Lambda_L,\rho,h}^a(\omega)=\frac{2^{\mathcal{K}\left(\Lambda_L^a, (\omega\rho)_{\Lambda^a_L}\right)}}{\tilde{Z}^a_{L,\rho,h}}\prod_{e\in\mathscr{B}(\Lambda^a_L)}(1-e^{-2\beta_c})^{\omega(e)}(e^{-2\beta_c})^{1-\omega(e)}\nonumber\\
\times\prod_{e\in\mathscr{E}(\Lambda^a_L)}(1-e^{-2a^{15/8}h})^{\omega(e)}(e^{-2a^{15/8}h})^{1-\omega(e)},
\end{align}
where $(\omega\rho)_{\Lambda^a_L}$ denotes the configuration which coincides with $\omega$ on $\mathscr{B}(\Lambda^a_L)\cup\mathscr{E}(\Lambda^a_L)$ and with $\rho$ on $\overline{\mathscr{B}}\left((\Lambda^a_L)^C\right)\cup\mathscr{E}\left((\Lambda^a_L)^C\right)$, $\mathcal{K}\left(\Lambda_L^a, (\omega\rho)_{\Lambda^a_L}\right)$ denotes the number of clusters in $(\omega\rho)_{\Lambda^a_L}$ which intersect $\Lambda_L^a$ and do not contain $g$, and $\tilde{Z}^a_{L,\rho,h}$ is the partition function. An edge $e$ is said to be \textit{open} if $\omega(e)=1$, otherwise it is said to be \textit{closed}. $\mathbb{P}_{\Lambda_L,\rho,h}^a$ is also called the \textit{random-cluster measure} (with cluster weight $q=2$) at $\beta_c$ on $\Lambda^a_{L}$ with boundary condition $\rho$ and with external field $a^{\frac{15}{8}}h\geq 0$. $\mathbb{P}^a_{\Lambda_L,f, h}$ (respectively, $\mathbb{P}^a_{\Lambda_L,w,h}$) denotes the probability measure with free (respectively, wired) boundary condition, i.e., $\rho\equiv 0$ (respectively, $\rho\equiv 1$) in \eqref{eqFKdef}. Below we will also consider FK measures $\mathbb{P}^a_{D,\rho,h}$ for more general domains $D\subseteq\mathbb{R}^2$, defined in the obvious way.

It is also known that $\mathbb{P}_{\Lambda_L,\rho,h}^a$ has a unique infinite volume limit as $L\rightarrow\infty$, which we denote by $\mathbb{P}_h^a$. Again this limiting measure does not depend on the choice of boundary condition. The reader may refer to \cite{Gri06} for more details in the case $h=0$; the proof for $h>0$ is similar.

\subsection{Basic properties}
The Edwards-Sokal coupling \cite{ES88}, based on the Swendsen-Wang algorithm \cite{SW87}, is a coupling of the Ising model and FK percolation. Let $\hat{\mathbb{P}}_h^a$ be such a coupling measure of $P_h^a$ and $\mathbb{P}_h^a$ defined on $\{-1,+1\}^V\times\{0,1\}^E$. The marginal of $\hat{\mathbb{P}}_h^a$ on $\{-1,+1\}^{a\mathbb{Z}^2}$ is $P_h^a$, and the marginal of $\hat{\mathbb{P}}_h^a$ on $\{0,1\}^E$ is $\mathbb{P}_h^a$. The conditional distribution of the Ising spin variables given a realization of the FK bonds can be realized by tossing independent fair coins --- one for each FK-open cluster not containing $g$ --- and then setting $\sigma_x$ for all vertices $x$ in the cluster to $+1$ for heads and $-1$ for tails. For $x$ in the ghost cluster, $\sigma_x= +1$ (for $h>0$). We note that a coupling for $h\neq 0$ between \emph{internal} FK edges and spin variables is given in Lemma \ref{lemtanh} and Proposition \ref{propFKcon} below.

For any $u,v\in V$, we write $u\longleftrightarrow v$ for the event that there is a path of FK-open edges that connects $u$ and $v$, i.e., a path $u=z_0, z_1, \ldots, z_n=v$ with $e_i=\{z_i, z_{i+1}\}\in E$ and $\omega(e_i)=1$ for each $0\leq i<n$. For any $A, B\subseteq V$, we write $A\longleftrightarrow B$ if there is some $u\in A$ and $v\in B$ such that $u\longleftrightarrow v$. $A\centernot\longleftrightarrow B$ denotes the complement of $A\longleftrightarrow B$. The following identity, immediate from the coupling, is essential.

\begin{lemma}\label{lemES}
\begin{equation}\label{eqES}
\langle\sigma_x;\sigma_y\rangle_{a,h}=\mathbb{P}_h^a(x\longleftrightarrow y)-\mathbb{P}_h^a(x\longleftrightarrow g)\mathbb{P}_h^a(y\longleftrightarrow g).
\end{equation}
\end{lemma}


Let $\mathbb{P}^a:=\mathbb{P}^a_{h=0}$. By standard comparison inequalities for FK percolation (Proposition 4.28 in \cite{Gri06}), one has
\begin{lemma}\label{lem3}
For any $h\geq 0$, $\mathbb{P}_h^a$ stochastically dominates $\mathbb{P}^a$.
\end{lemma}




The following lemma is about the one-arm exponent for FK percolation with $h=0$. The proof is a direct consequence of Wu's result \cite{Wu66,MW73} and the RSW-type result \cite{DCHN11} (see also Lemma 5.4 of \cite{DCHN11} for a different proof).
\begin{lemma}\label{lem1arm}
There exists a constant $C_1$ independent of $a$ such that for all $a\leq 1$ and for any boundary condition $\rho\in\{0,1\}^{\overline{\mathscr{B}}\left((\Lambda^a_1)^C\right)\cup\mathscr{E}\left((\Lambda^a_1)^C\right)}$,
\[\mathbb{P}^a_{\Lambda_1,\rho,h=0}(0\longleftrightarrow\partial_{in}\Lambda_1^a)\leq C_1a^{1/8}.\]
\end{lemma}

Let $D\subseteq \mathbb{R}^2$ be bounded,  and $D^a:=a\mathbb{Z}^2\cap D$ be the $a$-approximation of $D$. For any $\omega\in\{0,1\}^{\mathscr{B}(D^a)}$, let $\mathscr{C}(D^a,\omega)$ denote the set of clusters of $\omega$; for a $\mathcal{C}\in\mathscr{C}(D^a,\omega)$, let $|\mathcal{C}|$ denote the number of vertices in $\mathcal{C}$. Then we have
\begin{lemma}\label{lemtanh}
For any  $\omega\in\{0,1\}^{\mathscr{B}(D^a)}$, suppose $\mathscr{C}(D^a,\omega)=\{\mathcal{C}_1,\mathcal{C}_2,\ldots\}$ where $\mathcal{C}_i$'s are distinct. Then for any $\mathcal{C}_i\in \mathscr{C}(D^a,\omega)$
\begin{align}\label{eqtanh}
 \mathbb{P}_{D,f,h}^a(\mathcal{C}_i\longleftrightarrow g|\omega)=\tanh(ha^{15/8}|\mathcal{C}_i|).
\end{align}
Moreover, conditioned on $\omega$, the events $\{\mathcal{C}_i\longleftrightarrow g\}$ are mutually independent.
\end{lemma}
\begin{proof}
This follows from the proof of the next proposition.
\end{proof}

\begin{proposition}\label{propFKcon}
The Radon-Nikodym derivative of $\tilde{\mathbb{P}}_{D,f,h}^a$, the marginal of $\mathbb{P}^a_{D,f,h}$ on $\mathscr{B}(D^a)$, with respect to $\mathbb{P}_{D,f,h=0}^a$ is
\begin{equation}\label{eqRN}
\frac{d\tilde{\mathbb{P}}_{D,f,h}^a}{d\mathbb{P}_{D,f,h=0}^a}(\omega)=\frac{\prod_{\mathcal{C}\in \mathscr{C}(D^a,\omega)}\cosh(ha^{15/8}|\mathcal{C}|)}{\mathbb{E}_{D,f,h=0}^a\left[\prod_{\mathcal{C}\in \mathscr{C}(D^a,\cdot)}\cosh(ha^{15/8}|\mathcal{C}|)\right]} \; \text{ for each } \omega\in\{0,1\}^{\mathscr{B}(D^a)},
\end{equation}
where $\mathbb{E}_{D,f,h=0}^a$ is the expectation with respect to $\mathbb{P}_{D,f,h=0}^a$. Let $\hat{\mathbb{P}}^a_{D,f,h}$ be the Edwards-Sokal coupling of $\mathbb{P}^a_{D,f,h}$ and its corresponding Ising measure. For any $\mathcal{C}\in \mathscr{C}(D^a,\omega)$, let $\sigma(\mathcal{C})$ be the spin value of the cluster assigned by the coupling. Then we have, for $\omega\in\{0,1\}^{\mathscr{B}(D^a)}$,
\[\hat{\mathbb{P}}_{D,f,h}^a(\sigma(\mathcal{C}_i)=+1|\omega)=\tanh(ha^{15/8}|\mathcal{C}_i|)+\frac{1}{2}\left(1-\tanh(ha^{15/8}|\mathcal{C}_i|)\right),\]
\[\hat{\mathbb{P}}_{D,f,h}^a(\sigma(\mathcal{C}_i)=-1|\omega)=\frac{1}{2}\left(1-\tanh(ha^{15/8}|\mathcal{C}_i|)\right).\]
Moreover, conditioned on $\omega$, the events $\{\sigma(\mathcal{C}_i)=+1\}$ are mutually independent.
\end{proposition}

\begin{proof}
It is not hard to show that (see e.g. pp. 447-448 of \cite{Ale98}) for each $\omega\in\{0,1\}^{\mathscr{B}(D^a)}$
\begin{equation}\label{eqFKweight}
\mathbb{P}_{D,f,h}^a(\omega)\propto\left(1-e^{-2\beta_c}\right)^{o(\omega)}\left(e^{-2\beta_c}\right)^{c(\omega)}\prod_{\mathcal{C}\in \mathscr{C}(D^a,\omega)}\left((1-e^{-2ha^{15/8}|\mathcal{C}|})+2e^{-2ha^{15/8}|\mathcal{C}|}\right),
\end{equation}
where $o(\omega)$ and $c(\omega)$ denote the number of open and closed edges of $\omega$ respectively. So \eqref{eqRN} follows from \eqref{eqFKweight}, \eqref{eqFKdef} and the fact $\tilde{\mathbb{P}}_{D,f,h}^a(\omega)=\mathbb{P}_{D,f,h}^a(\omega)$. \eqref{eqFKweight} also gives, for any $\mathcal{C}_i, \mathcal{C}_j\in \mathscr{C}(D^a,\omega)$ with $i\neq j$,
\[\mathbb{P}_{D,f,h}^a(\mathcal{C}_i\longleftrightarrow g|\omega)=\frac{1-e^{-2ha^{15/8}|\mathcal{C}_i|}}{(1-e^{-2ha^{15/8}|\mathcal{C}_i|})+2e^{-2ha^{15/8}|\mathcal{C}_i|}}=\tanh(ha^{15/8}|\mathcal{C}_i|),\]
\[\mathbb{P}_{D,f,h}^a(\mathcal{C}_i\longleftrightarrow g, \mathcal{C}_j\longleftrightarrow g|\omega)=\tanh(ha^{15/8}|\mathcal{C}_i|)\tanh(ha^{15/8}|\mathcal{C}_j|),\]
with a similar product expression for the intersection of three or more of the events $\{\mathcal{C}_i\longleftrightarrow g\}$. Hence, conditioned on $\omega$, these events are mutually independent. The rest of the proof follows from the Edwards-Sokal coupling.
\end{proof}
\begin{remark}
This type of analysis can be extended to the continuum like what is done in~\cite{CCK17} for $h=0$, with the continuum analog of the coupling in Proposition \ref{propFKcon} valid also for $h>0$. See Theorem 2 of \cite{CJN19} for such an extension.
\end{remark}

\section{Exponential decay on the lattice}\label{seclargeh}
\subsection{Exponential decay for long FK-open paths not connected to the ghost}

\begin{figure}
\begin{center}
\includegraphics[width=0.6\textwidth]{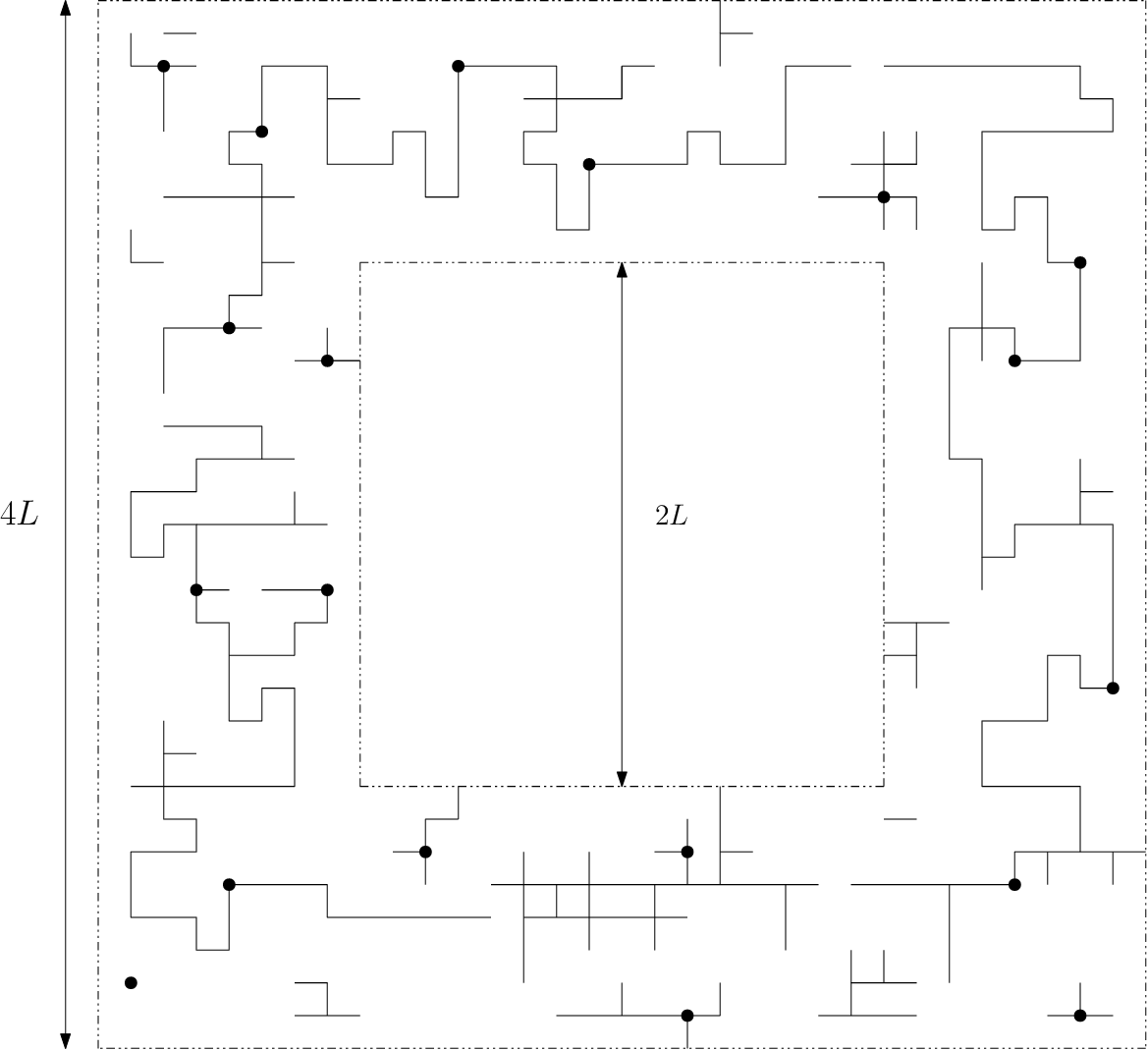}
\caption{An illustration of the event $G(a,L)$ with $a=1$ and $L=8$. FK-open edges inside the annulus are represented by \emph{solid segments} and the vertices of $\{z\in A^a(L,2L):\{z,g\} \text{ is open}\}$ are represented by \emph{black dots}.}\label{fig:G}
\end{center}
\end{figure}

Let $A(1,L):=[-L,L]^2\setminus [-1,1]^2$ be the annulus with inner radius $1$ and outer radius $L>1$. Let $A^a(1,L)$ be the $a$-approximation of $A(1,L)$. When $a$ is small, the interior boundary of $A^a(1,L)$ (i.e., $\{z\in  A^a(1,L): z\text{ has a nearest neighbor in  } \big(A^a(1,L)\big)^C\}$) naturally splits into a portion contained in $\big(-(1+L)/2,(1+L)/2\big)^2$ (denoted by $\partial_1A^a(1,L)$) and one contained in the complement of $\big(-(1+L)/2,(1+L)/2\big)^2$ (denoted by $\partial_2 A^a(1,L)$).
\begin{definition}
Let $F(a,L)$ be the event that in $A^a(1,L)$ there is an FK-open path from $\partial_1 A^a(1,L)$ to $\partial_2 A^a(1,L)$ consisting of vertices not connected via $A^a(1,L)$ to $g$.
\end{definition}
Similarly, let $A(L,2L)$ be the annulus with inner radius $L$ and outer radius $2L$ and $A^a(L,2L)$ be its $a$-approximation. We will consider circuits in the annulus --- i.e., nearest neighbor self-avoiding paths of vertices that end up at their starting vertex.
\begin{definition}
Let $G(a,L)$ denote the event that there is a circuit of vertices surrounding $[-L,L]^2$ in the annulus $A^a(L,2L)$ with each vertex in the circuit connected to $g$ via $A^a(L,2L)$; see Figure \ref{fig:G}.
\end{definition}
Denote the complement of $G(a,L)$ by $G^{comp}(a,L)$. The following proposition shows that the probabilities of $F(a,L)$ and $G^{comp}(a,L)$ decay exponentially.

\begin{proposition}\label{propFK}
For any $h>0$, there exist $\epsilon_0=\epsilon_0(h)\in (0,\infty)$ and $N_1=N_1(h)\in[2,\infty)$ such that for all $a\leq \epsilon_0$, $L\geq N_1$, for any boundary condition $\rho_1$ on $A^a(1,L)$ and any boundary condition $\rho_2$ on $A^a(L,2L)$, we have
\begin{eqnarray}
&&\mathbb{P}^a_{A(1,L),\rho_1,h}\left(F(a,L)\right)\leq e^{-C_1(h)L}\label{eqFK1}\\
&&\mathbb{P}^a_{A(L,2L),\rho_2,h}\left(G^{comp}(a,L)\right)\leq e^{-C_1(h)L}\label{eqFK2}
\end{eqnarray}
where $C_1(h)\in (0,\infty)$ only depends on h.
\end{proposition}

Before proving Proposition \ref{propFK}, we state and prove several lemmas. The first gives a useful property of CLE$_{16/3}$ and its related conformal measure ensemble; the idea of such coupled loop and measure ensembles originated in \cite{CN09}. Let $\Lambda_{3,1}:=[0,3]\times[0,1]$ and $\Lambda^a_{3,1}$ be its $a$-approximation. By Theorem \ref{thmconvcle} in Appendix \upperRomannumeral{1}, in the scaling limit $a\downarrow 0$, the loop ensemble of boundaries of ($h=0$) FK-open clusters in $\Lambda_{3,1}^a$ with free boundary condition converges in distribution to CLE$_{16/3}$ in $\Lambda_{3,1}$. From Theorem \ref{thmconvcme}, we know that the joint law of the collection of boundaries of FK-open clusters and the collection of normalized counting measures (with normalization $a^{15/8}$) of the FK-open clusters converges in distribution, in the same limit $a\downarrow 0$, to the joint law of CLE$_{16/3}$ and a collection of limiting counting measures (a conformal measure ensemble). Let $\mathbb{P}_{\Lambda_{3,1}}$ denote the latter joint law (i.e., in the scaling limit).

\begin{lemma}\label{lemcle}
Let $\mathbb{P}_{\Lambda_{3,1}}$ be the joint law of nested CLE$_{16/3}$ and CME$_{16/3}$ in $\Lambda_{3,1}$ with free boundary condition. Let $E(K; \eta)$ for $K\in\mathbb{N}$ and $\eta>0$ be the event that there is a sequence of $K$ or fewer loops (say, $L_1, \dots, L_k$ with $k\leq K$) such that the total mass of the limiting counting measure corresponding to $L_i$ is $\geq \eta$ for each $i$ and
\begin{equation}\label{eqCLE}
\mathrm{dist}(L_1,\{0\}\times[0,1])=0,\mathrm{dist}(L_i,L_{i+1})=0 \text{ for each } 1\leq i\leq k-1, \mathrm{dist}(L_k,\{3\}\times[0,1])=0,
\end{equation}
where $\mathrm{dist}(\cdot,\cdot)$ denotes Euclidean distance; see Figure \ref{fig:E(K)}. Then there is a choice of $\eta_K\rightarrow 0$ such that
\[\lim_{K\rightarrow\infty}\mathbb{P}_{\Lambda_{3,1}}(E(K; \eta_K))=1.\]
\end{lemma}

\begin{figure}
\begin{center}
\includegraphics[width=0.8\textwidth]{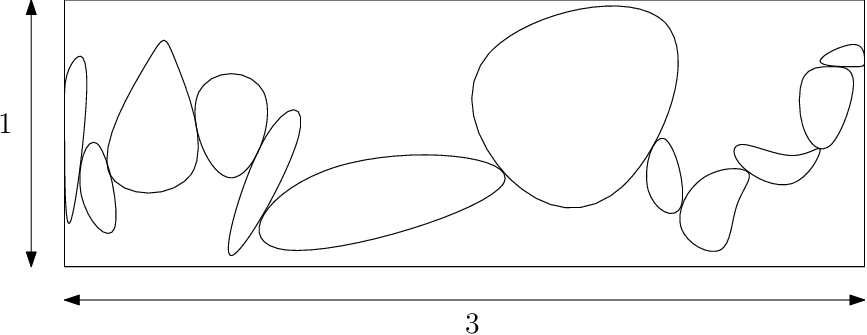}
\caption{An illustration of the event $E(K)$.}\label{fig:E(K)}
\end{center}
\end{figure}

\begin{proof}
Let $\psi:\Lambda_{3,1}\rightarrow \mathbb{D}:=\{z:|z|\leq 1\}$ be the conformal map with $\psi\left((3/2,1/2)\right)=0$ and $\psi^{\prime}\left((3/2,1/2)\right)>0$. Let $\Gamma_1:=\psi(\{0\}\times[0,1])$ and $\Gamma_2:=\psi(\{3\}\times [0,1])$. We first prove that, with probability $1$, CLE$_{16/3}$ in $\mathbb{D}$ contains a finite sequence of loops, $L_1, \ldots, L_k$, such that
\begin{equation}\label{eqCLE3}
\mathrm{dist}(L_1,\Gamma_1)=0,\mathrm{dist}(L_i,L_{i+1})=0 \text{ for any } 1\leq i\leq k-1, \mathrm{dist}(L_k,\Gamma_2)=0.
\end{equation}
Then, the conformal invariance of CLE$_{16/3}$ implies that a finite sequence satisfying \eqref{eqCLE} exists in $\Lambda_{3,1}$ with $\mathbb{P}_{\Lambda_{3,1}}$-probability $1$.

Our argument is inspired by the proof of Lemma 9.3 in \cite{SW12}. Let $L^*$ be the outermost loop containing $0$, and let $D^*$ be the connected component of $\mathbb{D} \setminus L^*$ containing $0$. Let $O_1$ (respectively, $\mathcal{O}_1$) be the union (respectively, collection) of all loops that touch $\Gamma_1$, then clearly $O_1\neq \emptyset$ with probability $1$. If $L^*\in \mathcal{O}_1$, then we stop; otherwise we let $D_1$ be the connected component of $\mathbb{D}\setminus O_1$ containing $0$. In this case, the conformal radius $\rho_1$ of $D_1$ seen from $0$ has a strictly positive probability to be strictly smaller than $1$, and the harmonic measure of $\partial_1 := O_1\cup \Gamma_1$ from $0$ in $\mathbb{D}$ is not smaller than the harmonic measure of $\Gamma_1$ in $\mathbb{D}$ from $0$. We now consider the CLE$_{16/3}$ in $D_1$, and we let $O_2$ (respectively, $\mathcal{O}_2$) be the union (respectively, collection) of all loops that touch $\partial_1$. If $L^*\in \mathcal{O}_2$, then we stop; otherwise we let $D_2$ be the connected component of $D_1\setminus O_2$ containing $0$, and we interate the procedure. After $i$ steps, the conformal radius $\rho_i$ of $D_i$ seen from $0$ is stochastically smaller than a product of $n$ i.i.d. copies of $\rho_1$. Since the conformal radius of $D^*$ from $0$ is strictly positive with probability $1$, this shows that, with probability $1$, $L^*$ is reached in a finite number of steps. Hence, there exists almost surely a finite sequence of loops $L_1,\dots,L_n$ (with $L_i\in \mathcal{O}_i$ for each $i<n$) such that
\[\mathrm{dist}(L_1,\Gamma_1)=0,\mathrm{dist}(L_i,L_{i+1})=0 \text{ for any } 1\leq i\leq n-1, L_n=L^*.\]

By the same argument, one can find a finite sequence of loops (say, $L_1^{\prime}, \dots, L^{\prime}_j$) such that
\[\mathrm{dist}(L^{\prime}_1,\Gamma_2)=0,\mathrm{dist}(L^{\prime}_i,L^{\prime}_{i+1})=0 \text{ for any } 1\leq i\leq j-1, L^{\prime}_j=L^*.\]
The sequence of loops $L_1, \dots, L_{n-1}, L^*, L^{\prime}_{j-1}, \dots, L_1^{\prime}$ satisfies \eqref{eqCLE3} with $k=n+j-1$, and the proof is concluded by noting that the mass of each limiting counting measure associated to a loop in that sequence is almost surely strictly positive (see Corollary~\ref{corpostmass} in Appendix \upperRomannumeral{1}).
\end{proof}

\begin{remark}
It is clear that in Lemma \ref{lemcle} without loss of generality  we may take $\eta_K=1/K$ which we henceforth do and then define the event $E(K):=E(K;1/K)$.
\end{remark}

For $N\in\mathbb{N}$, let $\Lambda_{3N,N}:=[0,3N]\times[0,N]$ and $\Lambda^a_{3N,N}$ be its $a$-approximation. By the conformal invariance of CLE$_{16/3}$, the conformal covariance of the limiting counting measures \cite{CCK17}, and Lemma \ref{lemcle}, we have
\begin{lemma}\label{lemcle2}
For $N\in\mathbb{N}$, let $\mathbb{P}_{\Lambda_{3N,N}}$ be the joint law of nested CLE$_{16/3}$ and CME$_{16/3}$ in $\Lambda_{3N,N}$ with free boundary condition. Let $E(K,N)$ be the event that there is a sequence of $K$ or fewer loops (say, $L_1, \ldots, L_k$ with $k\leq K$) such that the total mass of the limiting counting measure corresponding to $L_i$ is $\geq N^{15/8}/K$ for each $i$ and
\begin{eqnarray*}
&\mathrm{dist}(L_1,\{0\}\times[0,N])=0,\mathrm{dist}(L_i,L_{i+1})=0 \text{ for each } 1\leq i\leq k-1,\\ &\mathrm{dist}(L_k,\{3N\}\times[0,N])=0.
\end{eqnarray*}
Then for any $\epsilon>0$, there exists $K(\epsilon)<\infty$ such that
\[\mathbb{P}_{\Lambda_{3N,N}}(E(K,N))>1-\epsilon \text{ for all }K\geq K(\epsilon).\]
\end{lemma}
\begin{proof}
Using the conformal Markov property of CLE, CLE$_{16/3}$ in $\Lambda_{3,1}$ can be obtained from a (nested) full plane CLE$_{16/3}$ as follows. Consider the outermost loop $L$ in the unit disc $\mathbb D$ surrounding the origin and let $D_0$ denote the connected component of $\mathbb D \setminus L$ containing the origin. Conditioned on $L$, the loops inside $D_0$ are distributed like a CLE$_{16/3}$ in $D_0$. Therefore, CLE$_{16/3}$ inside $\Lambda_{3,1}$ can be obtained from CLE$_{16/3}$ inside $D_0$ by a conformal map from $D_0$ to $\Lambda_{3,1}$. Together with the measurability of the limiting counting measures with respect to the CLE loops (Corollary \ref{cormeas} in Appendix \upperRomannumeral{1}), this shows that the limiting counting measures inside $\Lambda_{3,1}$ scale like the full plane versions, so that one can apply Theorem 2.4 of \cite{CCK17}. The lemma now follows immediately from Lemma 5 by considering a scale transformation from $\Lambda_{3,1}$ to $\Lambda_{3N,N}$.
\end{proof}

The next lemma says that on $a\mathbb{Z}^2$, with high probability, we can find (for $h=0$) a finite sequence of FK clusters in $\Lambda_{3N,N}$ whose concatenation almost forms an open crossing of $\Lambda_{3N,N}$ in the horizontal direction.
\begin{lemma}\label{lemnec1}
For $N\in\mathbb{N}$, let $E^a(K,N)$ be the event that there exists a sequence $\mathcal{C}_1,\ldots,\mathcal{C}_k$ of FK-open clusters in $\Lambda_{3N,N}^a$ such that $k\leq K$, $|\mathcal{C}_i|\geq N^{15/8}a^{-15/8}/K$ for each $i$ and
\begin{eqnarray*}
&\mathrm{dist}(\mathcal{C}_1,\{0\}\times[0,N])\leq a, \mathrm{dist}(\mathcal{C}_i,\mathcal{C}_{i+1})=a \text{ for every } 1\leq i\leq k-1,\\ &\mathrm{dist}(\mathcal{C}_k,\{3N\}\times[0,N])\leq a.
\end{eqnarray*}
Then for any $\epsilon>0$, there exists $K(\epsilon)<\infty$ such that
\[\liminf_{a\downarrow 0}\mathbb{P}^a_{\Lambda_{3N,N},f,0}(E^a(K,N))>1-\epsilon, \text{ for all }K\geq K(\epsilon).\]
\end{lemma}
\begin{proof}
Let $L_2$ and $L_3$ be distinct CLE$_{16/3}$ loops inside $\Lambda_{3N,N}$ such that $\mathrm{dist}(L_2,L_3)=0$. Because of the convergence of the collection of the lattice boundaries of critical FK clusters to CLE$_{16/3}$ (Theorem \ref{thmconvcle}), there is a coupling between FK percolation in $\Lambda_{3N,N}$ and CLE$_{16/3}$ such that the pair $(L_2^a,L_3^a)$ of lattice boundaries of two FK-open clusters converges a.s. to $(L_2,L_3)$. Under this coupling, we claim that the probability of $\mathrm{dist}(L_2^a,L_3^a)\leq a$ tends to $1$ as $a\downarrow 0$.
Indeed, it is easy to see that if $\mathrm{dist}(L_1^a,L_2^a) > a$, then there is a 6-arm event of type $(100100)$ (see page 4 of \cite{CDCH16} for the precise definition of this event). But by Theorem \ref{thm6arm}, the critical exponent for a 6-arm event of type $(100100)$ is strictly larger than $2$. It follows (see, e.g. the proof of Lemma~6.1 of \cite{CN06}) that the probability of seeing a 6-arm event anywhere goes to $0$ as $a\downarrow 0$. This completes the proof of the claim.

By a similar argument, using the fact that the exponent for a 3-arm event near a boundary is strictly larger than 1 (Corollary 1.5 of \cite{CDCH16}) and hence they do not occur as $a\downarrow 0$, one can prove that, if $L_1$ is a loop such that $\mathrm{dist}(L_1,\{0\}\times[0,N])=0$, then there is a coupling between FK percolation and CLE$_{16/3}$ in $\Lambda_{3N,N}$ such that the FK lattice boundary $L_1^a$ converges a.s. to $L_1$ and also that the probability that $\mathrm{dist}(L_1^a,\{0\}\times[0,N])\leq a$ tends to $1$.
Combining this and the previous claim with Theorem \ref{thmconvcme} and with Lemma \ref{lemcle2} above completes the proof of the lemma.
\end{proof}

Let $\Lambda_{3N,3N}:=[0,3N]\times[0,3N]$ and $A_{N,3N}$ be the annulus $\Lambda_{3N,3N}\setminus[N,2N]^2$, and let $\Lambda_{3N,3N}^a$ and $A_{N,3N}^a$ be their $a$-approximations respectively. Let $\mathcal{N}^a(K,N)$ be the event that there is a necklace consisting of open clusters in $A_{N,3N}^a$ surrounding $[N,2N]^2$. More precisely, $\mathcal{N}^a(K,N)$ is the event that there is a sequence of FK-open clusters in $A_{N,3N}^a$ (say $\mathcal{C}_1,\ldots,\mathcal{C}_k$ with $k\leq K$) such that
\begin{eqnarray*}
&\mathrm{dist}(\mathcal{C}_i,\mathcal{C}_{i+1})=a \text{ for each } 1\leq i\leq k-1, \mathrm{dist}(\mathcal{C}_k,\mathcal{C}_1)= a, \\
&|\mathcal{C}_i|\geq N^{15/8}a^{-15/8}/K \text{ for each } 1\leq i \leq k,
\end{eqnarray*}
and there is a circuit of vertices in $\cup_{i=1}^k \mathcal{C}_i$ surrounding $[N,2N]^2$. Then we have
\begin{lemma}\label{lemnec2}
For any $N\in\mathbb{N}$ and $\epsilon>0$, there exists $K_1(\epsilon)<\infty$ such that
\[\liminf_{a\downarrow 0}\mathbb{P}^a_{\Lambda_{3N,3N},f,0}(\mathcal{N}^a(K,N))>1-\epsilon, \text{ for all }K\geq K_1(\epsilon).\]
\end{lemma}
\begin{proof}
We use a standard argument in the percolation literature --- see, e.g., Figure 3 in \cite{BC10} --- as follows. It is easy to show that $\mathcal{N}^a(K,N)$ contains the intersection of four events which are rotated and/or translated versions of $E^a(K/4,N)$. Note that $E^a(K/4,N)$ is an increasing event. So the lemma follows from the FKG inequality and Lemma \ref{lemnec1}.
\end{proof}

Next, we consider FK percolation with external field $a^{15/8}h$. We say $A_{N,3N}^a$ is \textit{good} if there is a sequence of open clusters in $A_{N,3N}^a$ (say $\mathcal{C}_1,\ldots,\mathcal{C}_k$ for some $k\in\mathbb{N}$) such that
\[\mathrm{dist}(\mathcal{C}_i,\mathcal{C}_{i+1})=a \text{ for each } 1\leq i\leq k-1, \mathrm{dist}(\mathcal{C}_k,\mathcal{C}_1)= a, \mathcal{C}_i\longleftrightarrow g \text{ for each }i\]
and there is a circuit of vertices in $\cup_{i=1}^k \mathcal{C}_i$ surrounding $[N,2N]^2$.
\begin{lemma}\label{lemlss}
Given any $h>0$ and $\epsilon>0$, there exist $N_0\in[1,\infty)$ and $\epsilon_0\in(0,\infty)$ such that for $N\geq N_0$ and $a\leq \epsilon_0$,
\[\mathbb{P}^a_{\Lambda_{3N,3N},f,h}(A_{N,3N}^a \mathrm{~is~good})\geq 1-\epsilon.\]
\end{lemma}
\begin{proof}
For any fixed $\epsilon>0$, by Lemma \ref{lemnec2}, there exist $K_0,\epsilon_0>0$ such that
\[\mathbb{P}^a_{\Lambda_{3N,3N},f,0}(\mathcal{N}^a(K_0,N))>1-\epsilon/2 \text{ for all } a\leq\epsilon_0, N\in\mathbb{N}.\]
So, by Lemma \ref{lem3},
\begin{equation}
\mathbb{P}^a_{\Lambda_{3N,3N},f,h}(\mathcal{N}^a(K_0,N))>1-\epsilon/2 \text{ for all } a\leq\epsilon_0, N\in\mathbb{N}.\label{eqNaK}
\end{equation}
Lemma \ref{lemtanh} implies that for each $\mathcal{C}_i$ from the definition of $\mathcal{N}^a(K)$,
\[\mathbb{P}^a_{\Lambda_{3N,3N},f,h}\left(\mathcal{C}_i\longleftrightarrow g|\mathcal{N}^a\left(K_0,N\right)\right)=\tanh(ha^{15/8}|\mathcal{C}_i|)\geq\tanh(hN^{15/8}/K_0).\]
Therefore,
\begin{align*}
&\mathbb{P}^a_{\Lambda_{3N,3N},f,h}(A_{N,3N}^a \text{ is good })\\
&\geq \mathbb{P}^a_{\Lambda_{3N,3N},f,h}(\mathcal{C}_i\longleftrightarrow g \text{ for each } i|\mathcal{N}^a\left(K_0,N\right))\mathbb{P}^a_{\Lambda_{3,3},f,h}(\mathcal{N}^a(K_0,N))\\
&\geq(\tanh(hN^{15/8}/K_0))^{K_0}(1-\epsilon/2)\geq 1-\epsilon \text{ if } a\leq \epsilon_0 \text{ and }N \text { is large},
\end{align*}
where the second inequality follows from Lemma \ref{lemtanh} and \eqref{eqNaK}.
\end{proof}

We are ready to prove Proposition \ref{propFK}. Our argument is similar to ones appearing elsewhere in the percolation literature --- see, e.g., the proof of Lemma 5.3 in \cite{BC10}.
\begin{proof}[Proof of Proposition \ref{propFK}]
We first consider FK percolation on $a\mathbb{Z}^2$. For each $z=(z_1,z_2)\in \mathbb{Z}^2$, let
\[A_{N,3N}(z):=N\times(z_1-3/2,z_2-3/2)+A_{N,3N}\]
and $A_{N,3N}^a(z)$ be its $a$-approximation. We define whether $A_{N,3N}^a(z)$ is good (or not) by the translation of the definition for $A_{N,3N}^a$ and then define a family of random variables $\{Y_z, z\in\mathbb{Z}^2\}$ such that $Y_z=1$ if $A_{N,3N}^a(z)$ is good and $Y_z=0$ otherwise. Note that the worst boundary condition for the event $\{A_{N,3N}^a \text{ is good}\}$ is the free boundary condition on the boundary of $\Lambda_{3N,3N}^a$. Then by Theorem 0.0 of \cite{LSS97} and Lemma \ref{lemlss}, $\{Y_z, z\in\mathbb{Z}^2\}$ stochastically dominates a family of i.i.d. random variables $\{Z_z, z\in\mathbb{Z}^2\}$ such that $P(Z_z=1)=\pi(\epsilon_0,N_0)$ and $P(Z_z=0)=1-\pi(\epsilon_0,N_0)$ where $\pi(\epsilon_0,N_0)$ can be made arbitrarily close to $1$ by choosing $\epsilon_0$ small and $N_0$ large.

We note that if $A_{N,3N}^a$ is good then there is a circuit of vertices surrounding $[N,2N]^2$ in $A_{N,3N}^a$ with each vertex in this circuit connected to $g$ in $A_{N,3N}^a$. Such a circuit prevents the existence of an FK-open path from the inner boundary $\partial_1 A_{N,3N}^a$ to the outer boundary $\partial_2 A_{N,3N}^a$ whose cluster does not contain $g$. This means that, whenever $Y_z=1$, there is no such FK-open path from $\partial_1A_{N,3N}^a(z)$ to $\partial_2 A_{N,3N}^a(z)$ whose cluster does not contain $g$. But whenever $F(a,L)$ occurs and $N\geq 2$, there is a nearest neighbor path (say $\gamma$) on $\mathbb{Z}^2$ starting at $0$ and reaching at least distance $L/N$ away from $0$ such that $Y_z=0$ for each $z\in \gamma$. Pick $\epsilon_0>0$ and $N_0\geq 2$ such that $\pi(\epsilon_0,N_0)$ is larger than the critical probability of site percolation on $\mathbb{Z}^2$. Note that $\epsilon_0$ and $N_0$ only depend on $h$. We fix $N=N_1=N_0$ in the rest of the proof of \eqref{eqFK1}. Then Theorem 6.75 of \cite{Gri99} (actually that theorem is for bond percolation but the proof also applies to site percolation) implies that there exists a finite constant $\tilde{C}_1(h)$ such that
\[\mathbb{P}^a_{A(1,L),\rho_1,h}(F(a,L))\leq e^{-\tilde{C}_1(h)L/N_0}=e^{-C_1(h)L}.\]

If $G^{comp}(a,L)$ occurs, then there is a $*$-path (i.e., one that can use both nearest neighbor and diagonal edges) from $\partial_1 A^a(L,2L)$ to $\partial_2 A^a(L,2L)$ such that each vertex in this path is not connected via $A^a(L,2L)$ to $g$. We note that if $A_{N,3N}^a$ is good then there is no such $*$-path (with the cluster of each vertex on the path not containing $g$) from the inner box to the outer boundary of $A_{N,3N}^a$. The rest of the proof of \eqref{eqFK2} is similar to that of \eqref{eqFK1} except that here we take $N_1(h)>N_0(h)$ in order to avoid a prefactor in \eqref{eqFK2}.
\end{proof}

\subsection{Exponential decay of $\langle \sigma_x; \sigma_y \rangle$}
Our goal in this subsection is to show the following
\begin{proposition}\label{propdisexp}
For any $h>0$, there exists $\epsilon_0=\epsilon_0(h)\in(0,1]$ such that for all  $a\leq \epsilon_0$
\begin{eqnarray*}
\langle \sigma_x;\sigma_y\rangle_{a,h}\leq C_4 a^{1/4} e^{-m_1(h)|x-y|} \text{ whenever }|x-y|>K_0(h) \text{ and }x,y\in a\mathbb{Z}^2,
\end{eqnarray*}
where $C_4\in(0,\infty)$ is universal and $m_1(h), K_0(h)\in(0,\infty)$ only depend on h.
\end{proposition}

Although we do not use it in our current proof, there is a nice BK-type inequality for Ising variables \cite{BG13} which can at least give partial results on exponential decay; perhaps a more careful use would give complete results.

Let $B(z,L):=z+\Lambda_L$ for $z\in\mathbb{R}^2$ and $L>0$ denote the square centered at $z$ (parallel to the coordinate axes) of side length $2L$. Recall that $P_h^a$ is the infinite volume measure for the Ising model on $a\mathbb{Z}^2$ at critical inverse temperature $\beta_c$ with external field $a^{15/8}h$. Let $P^a_{\vec{h}}$ be the same infinite volume measure except that the external field is $0$ in $B(x,1)\cup B(y,1)$. Let $\langle \cdot\rangle_{a,\vec{h}}$ be the expectation with respect to $P_{\vec{h}}^a$, and $\mathbb{P}^a_{\vec{h}}$ be the corresponding FK percolation measure.

For the rest of this section, for simplicity we assume $x,y\in a\mathbb{Z}^2$ are on the $x$-axis; otherwise one has to slightly modify choices of lengths of some squares by factors of $1/\sqrt{2}$. For ease of notation, we also suppress the superscript $a$ on various events defined below ($A^0, A_z^1, A_z^c, A_z^f$) even though these are all defined in the $a\mathbb{Z}^2$ setting; we keep the superscript $a$ in the various probability measures, such as $\mathbb{P}^a_{\vec{h}}$.

To bound $\langle \sigma_x;\sigma_y\rangle_{a,h}$, we first use the GHS inequality \cite{GHS70} to see that
\[\langle \sigma_x; \sigma_y\rangle_{a,h}\leq\langle \sigma_x; \sigma_y\rangle_{a,\vec{h}}.\]

Let $A^0:=\{x\longleftrightarrow y\centernot\longleftrightarrow g\}$, $A_z^1:=\{z\longleftrightarrow g\}$ for $z=x$ or $y$. Then the Edwards-Sokal coupling (like in Lemma \ref{lemES}) gives
\[\langle \sigma_x; \sigma_y\rangle_{a,\vec{h}}=\mathbb{P}^a_{\vec{h}}(A^0)+\mathbb{P}^a_{\vec{h}}(A^1_x\cap A^1_y)-\mathbb{P}^a_{\vec{h}}(A^1_x)\mathbb{P}^a_{\vec{h}}(A^1_y).\]
Now write $A^1_z$ for $z=x$ or $y$ as the disjoint partition $A^1_z=A^c_z\cup A^f_z$ ($c$ for close, $f$ for far) where
\begin{align*}
A^c_z:=\{&\text{there exists an FK-open path from } z, \text{ within }B(z,|x-y|/3), \text { to some } \\
&w \text{ with the edge from } w \text{ to } g \text{ open}  \}
\end{align*}
and $A^f_z:=A^1_z\setminus A^c_z$. Then we arrive at the following lemma.
\begin{lemma}\label{lemGHS}
\begin{equation}\label{eqcov}
\langle \sigma_x; \sigma_y\rangle_{a,h}\leq \langle \sigma_x; \sigma_y\rangle_{a,\vec{h}}= \mathbb{P}^a_{\vec{h}}(A^0)+D_{ff}+D_{fc}+D_{cf}+D_{cc},
\end{equation}
where for $u,v\in\{f,c\}$, $D_{u,v}:=\mathbb{P}^a_{\vec{h}}(A_x^u\cap A_y^v)-\mathbb{P}^a_{\vec{h}}(A_x^u)\mathbb{P}^a_{\vec{h}}(A_y^v)$.
\end{lemma}

Next, we show that each term on the RHS of \eqref{eqcov} decays exponentially with the desired power law factor $a^{1/4}$.
\begin{proposition}\label{propcov}
For any $h>0$, there exist $\epsilon_0=\epsilon_0(h)\in (0,1]$ and $N_1=N_1(h)\in[2,\infty)$ such that for all $a\leq \epsilon_0$, and $x,y\in a\mathbb{Z}^2$ with $|x-y|>3N_1$,
\[\mathbb{P}^a_{\vec{h}}(A^0)\leq C_2a^{1/4}e^{-C_3(h)|x-y|}, D_{uv}\leq C_2a^{1/4}e^{-C_3(h)|x-y|} \text{ for any }u,v \in\{f,c\},\]
where $C_2\in(0,\infty)$ is universal, and $C_3(h)\in (0,\infty)$ only depends on $h$.
\end{proposition}

\begin{proof}
The proofs for $\mathbb{P}^a_{\vec{h}}(A^0)$, $D_{ff}$, $D_{fc}$ and $D_{cf}$ are similar to each other. The proof for $D_{cc}$ is harder. Let $\epsilon_0=\epsilon_0(h)$ and $N_1=N_1(h)$ be the same as in Proposition \ref{propFK}.

\emph{(1) $\mathbb{P}^a_{\vec{h}}(A^0)$.}
In order for $A^0$ to occur there must be one arm events in both $B(x,1)$ and $B(y,1)$, and in the complement of $B(x,1)\cup B(y,1)$ there must be a (long) open path from $\partial_{ex} B(x,1)$ to $\partial_{ex} B(y,1)$ with the open cluster (within that complement) of the path not connected to the ghost. We will use Lemma \ref{lem1arm} twice to get $(C_1a^{1/8})^2$ and Proposition \ref{propFK} twice to get the exponential factor. More precisely, define $A^{0,z}$ and $\tilde{A}^{0,z}$ for $z=x$ or $y$ as $A^{0,z}:=\{z\longleftrightarrow\partial_{in}B(z,1)\}$ and $\tilde{A}^{0,z}$ be the event that there is an open path from $\partial_{ex}B(z,1)$ to $\partial_{in}B(z,|x-y|/2)$ with the open cluster of that path in $B(z,|x-y|/2)\setminus B(z,1)$ not connected to $g$. Then
\[A^0\subseteq A^{0,x}\cap \tilde{A}^{0,x}\cap\tilde{A}^{0,y}\cap A^{0,y}\]
and by taking the worst case boundary condition and using translation invariance, we have by using Lemma \ref{lem1arm} and Proposition \ref{propFK} (twice each):
\begin{eqnarray*}
\mathbb{P}^a_{\vec{h}}(A^0)&\leq& \mathbb{P}^a_{\vec{h}}(A^{0,x}\cap \tilde{A}^{0,x}\cap\tilde{A}^{0,y}\cap A^{0,y})\\
&\leq& [\mathbb{P}^a_{\Lambda_1,w,h=0}(0\longleftrightarrow \partial_{in}\Lambda_1^a)]^2\cdot[\sup_{\rho}\mathbb{P}^a_{A(1,|x-y|/2),\rho, h}(F(a,|x-y|/2))]^2\\
&\leq&(C_1 a^{1/8})^2(e^{-C_1(h)|x-y|/2})^2\\
&= & C_2 a^{1/4} e^{-C_3(h)|x-y|}
\end{eqnarray*}
with $C_2=C_1^2$ and $C_3(h)=C_1(h)$.

\emph{(2) $D_{ff}$.}
This proof is close to that for part \emph{(1)} because
\[A_z^f\subseteq \bar{A}_z^f:=\{z\longleftrightarrow \partial_{in} B(z,1)\}\cap \bar{\bar{A}}_z^f,\]
where $\bar{\bar{A}}^f_z$ denote the event that there exists a (long) open path connecting $\partial_{ex}B(z,1)$ to $\partial_{in}B(z,|x-y|/3)$ within the annulus $\overline{Ann}(z):=B(z,|x-y|/3)\setminus B(z,1)$ with the open cluster of that path (within that annulus) not connected to the ghost. This leads to
\[\mathbb{P}^a_{\vec{h}}(\bar{A}_z^f)\leq C_1 a^{1/8}e^{-C_1(h)|x-y|/3}.\]
More generally, by considering the worst boundary condition twice in the sense of
\[\theta_z:=\sup_{\rho}\mathbb{P}^a_{\overline{Ann}(z),\rho,\vec{h}}(\bar{\bar{A}}_z^f),\]
where the sup is over all (FK) boundary conditions on both parts of the boundary of $\overline{Ann}(z)$, and doing that both for $z=x$ and $z=y$, one gets the last inequality in
\[D_{ff}=\mathbb{P}^a_{\vec{h}}(A_x^f\cap A_y^f)-\mathbb{P}^a_{\vec{h}}(A_x^f)\mathbb{P}^a_{\vec{h}}(A_y^f)\leq\mathbb{P}^a_{\vec{h}}(A_x^f\cap A_y^f)\leq (C_1 a^{1/8}e^{-C_1(h)|x-y|/3})^2.\]

\emph{(3) $D_{fc}$ and $D_{cf}$.}
Clearly, $D_{fc}=D_{cf}$, so we only need to prove decay for $D_{fc}$.
Note that \[D_{fc}=\mathbb{P}^a_{\vec{h}}(A_x^f\cap A_y^c)-\mathbb{P}^a_{\vec{h}}(A_x^f)\mathbb{P}^a_{\vec{h}}(A_y^c)\leq\mathbb{P}^a_{\vec{h}}(A_x^f\cap A_y^c).\]
$A_x^f$ is treated as in the proof of part \emph{(2)} but $A_y^c$ is handled by noting that $A_y^c\subseteq\{y\longleftrightarrow \partial_{in}B(y,1)\}$. This leads to
\begin{align*}
D_{fc}&\leq \mathbb{P}^a_{B(x,1),w,h=0}(x\longleftrightarrow \partial_{in}B(x,1))\cdot\theta_x\cdot\mathbb{P}^a_{B(y,1),w,h=0}(y\longleftrightarrow \partial_{in}B(y,1))\\
&\leq C_1^2 a^{1/4} e^{-C_1(h)|x-y|/3}.
\end{align*}

\emph{(4) $D_{cc}$.} We have that
\[D_{cc}=\mathbb{P}^a_{\vec{h}}(A^c_x\cap A^c_y)-\mathbb{P}^a_{\vec{h}}(A^c_x)\mathbb{P}^a_{\vec{h}}(A^c_y)=\mathbb{P}^a_{\vec{h}}(A^c_y)[\mathbb{P}^a_{\vec{h}}(A_x^c|A^c_y)-\mathbb{P}^a_{\vec{h}}(A^c_x)].\]
Now by Lemma \ref{lem1arm},
\[\mathbb{P}^a_{\vec{h}}(A^c_y)\leq \mathbb{P}^a_{B(y,1),w,h=0}(y\longleftrightarrow \partial_{in}B(y,1))\leq C_1a^{1/8}.\]
We consider the worst case boundary condition on $\partial_{ex}B(x,2|x-y|/3)$ to get
\[\mathbb{P}^a_{\vec{h}}(A_x^c|A^c_y)-\mathbb{P}^a_{\vec{h}}(A^c_x)\leq \mathbb{P}^a_{2/3,w,\vec{h}}(A_x^c)-\mathbb{P}^a_{2/3,f,\vec{h}}(A_x^c),\]
where $\mathbb{P}^a_{2/3,w,\vec{h}}$ and $\mathbb{P}^a_{2/3,f,\vec{h}}$ refer to wired and free boundary conditions on $B(x,2|x-y|/3)$. As in Proposition \ref{propFK}, let $G=G(a,|x-y|/3)$ denote the event that there is a circuit of vertices surrounding $B(x,|x-y|/3)$ in the annulus $Ann(1/3,2/3):=B(x,2|x-y|/3)\setminus B(x,|x-y|/3)$ with each vertex in the circuit connected to $g$ within the annulus. Then
\begin{align*}
&\mathbb{P}^a_{2/3,w,\vec{h}}(A_x^c)-\mathbb{P}^a_{2/3,f,\vec{h}}(A_x^c)=\mathbb{P}^a_{2/3,w,\vec{h}}(A_x^c)-[\mathbb{P}^a_{2/3,f,\vec{h}}(G)\mathbb{P}^a_{2/3,f,\vec{h}}(A_x^c|G)\\
&~~~~+\mathbb{P}^a_{2/3,f,\vec{h}}(G^{comp})\mathbb{P}^a_{2/3,f,\vec{h}}(A_x^c|G^{comp})]\\
&\leq\mathbb{P}^a_{2/3,f,\vec{h}}(G)[\mathbb{P}^a_{2/3,w,\vec{h}}(A_x^c)-\mathbb{P}^a_{2/3,f,\vec{h}}(A_x^c|G)]+\mathbb{P}^a_{2/3,f,\vec{h}}(G^{comp})\mathbb{P}^a_{2/3,w,\vec{h}}(A_x^c).
\end{align*}
$\mathbb{P}^a_{2/3,f,\vec{h}}(A_x^c|G)$ corresponds roughly to a wired boundary condition on some random circuit which is \emph{inside} the wired boundary condition of $\mathbb{P}^a_{2/3,w,\vec{h}}$. Since $A_x^c$ is an increasing event, one expects that
\[\mathbb{P}^a_{2/3,w,\vec{h}}(A_x^c)-\mathbb{P}^a_{2/3,f,\vec{h}}(A_x^c|G)\leq 0\]
by some stochastic domination argument. Indeed, this inequality is proved in the next lemma.
Then, by Proposition \ref{propFK},
\begin{align*}
\mathbb{P}^a_{2/3,w,\vec{h}}(A_x^c)-\mathbb{P}^a_{2/3,f,\vec{h}}(A_x^c)&\leq\mathbb{P}^a_{2/3,f,\vec{h}}(G^{comp})\mathbb{P}^a_{2/3,w,\vec{h}}(A_x^c)\\
&\leq \mathbb{P}^a_{2/3,f,\vec{h}}(G^{comp})\mathbb{P}^a_{B(x,1),w,h=0}(x\longleftrightarrow \partial_{in}B(x,1))\\
&\leq C_1a^{1/8}e^{-C_1(h)|x-y|/3}.
\end{align*}
This concludes the proof.
\end{proof}

\begin{lemma}\label{lemstochdom}
Let $\mathbf{C}$ be any deterministic circuit of vertices in the annulus $Ann(1/3,2/3)$. Let $\tilde{A}_{\mathbf{C}}$ denote the event that each $x\in\mathbf{C}$ is connected to $g$ within the annulus and let $A_{\mathbf{C}}$ denote the event that $\mathbf{C}$ is the \emph{outermost} such circuit. Then for any increasing event $A$ in the interior of $\mathbf{C}$ (including edges to $g$),
\begin{equation}\label{eqsd}
\mathbb{P}^a_{2/3,f,\vec{h}}(A|A_{\mathbf{C}})\geq \mathbb{P}^a_{2/3,w,\vec{h}}(A).
\end{equation}
With $G=\cup_{\mathbf{C}}\tilde{A}_{\mathbf{C}}=\cup_{\mathbf{C}}A_{\mathbf{C}}$, it follows that for any increasing event $E$ in $B(x,|x-y|/3)$,
\begin{equation*}
\mathbb{P}^a_{2/3,f,\vec{h}}(E|G)\geq \mathbb{P}^a_{2/3,w,\vec{h}}(E).
\end{equation*}
\end{lemma}
\begin{remark}
 We note that the above lemma is not trivial because $G$ is a random, not a deterministic, set. We also point out
that the proof below shows that this lemma applies to quite general annuli, boundary conditions, and magnetic field profiles $h(x)\geq 0$ (as opposed to only $Ann(1/3,2/3)$, $f$ and $w$, and $\vec{h}$).
\end{remark}
\begin{proof}
For simplicity, we let $B$ denote the $a$-approximation of $B(0,2|x-y|/3)$ in this proof. Let $D$ be the interior of $\mathbf{C}$. The stochastic domination \eqref{eqsd} will follow from the stronger stochastic domination  that
\begin{equation}\label{eqsd1}
\mathbb{P}_{2/3,f,\vec{h}}^a(A|A_{\mathbf{C}})\geq \mathbb{P}^a_{D,w,\vec{h}}(A) \text{ for any increasing event }A \text{ in } D,
\end{equation}
since $\mathbb{P}^a_{D,w,\vec{h}}$ stochastically dominates $\mathbb{P}_{2/3,w,\vec{h}}^a$ on $D$. To prove \eqref{eqsd1}, it is sufficient to prove that the Radon-Nikodym derivative $d\mathbb{P}^a_{2/3,f,\vec{h}}(\cdot|A_{\mathbf{C}})/d\mathbb{P}^a_{D,w,\vec{h}}(\cdot)$ is an increasing function (in the FKG sense). In the following proof, $\omega_{out}$ is always in $\{0,1\}^{(\mathscr{B}(B)\setminus\mathscr{B}(D))\cup (\mathscr{E}(B)\setminus\mathscr{E}(D))}$. By the $\vec{h}$ replacing a constant $h$ version of \eqref{eqFKdef}, for any $\omega_{in}\in\{0,1\}^{\mathscr{B}(D)\cup\mathscr{E}(D)}$,
\begin{align}
\mathbb{P}^a_{2/3,f,\vec{h}}(\omega_{in}|A_{\mathbf{C}})\propto \sum_{\omega:=\omega_{in}\oplus \omega_{out}\in A_{\mathbf{C}}}2^{\mathcal{K}\left(B, (\omega\rho^0)_{B}\right)}\prod_{e\in\mathscr{B}(B)}(1-e^{-2\beta_c})^{\omega(e)}(e^{-2\beta_c})^{1-\omega(e)}\nonumber\\
\times\prod_{e\in\mathscr{E}(B)}(1-e^{-2a^{15/8}\vec{h}_e})^{\omega(e)}(e^{-2a^{15/8}\vec{h}_e})^{1-\omega(e)},\label{eqsd2}
\end{align}
where $\rho^0$ is the configuration with every edge closed and $\omega_{in}\oplus \omega_{out}$ denotes the configuration in $\{0,1\}^{\mathscr{B}(B)\cup\mathscr{E}(B)}$ whose open edges are all those from $\omega_{in}$ or (disjointly) from $\omega_{out}$. Also,
\begin{align}
\mathbb{P}^a_{D,w,\vec{h}}(\omega_{in}|A_{\mathbf{C}})\propto 2^{\mathcal{K}\left(D, (\omega_{in}\rho^1)_{D}\right)}\prod_{e\in\mathscr{B}(B)}(1-e^{-2\beta_c})^{\omega_{in}(e)}(e^{-2\beta_c})^{1-\omega_{in}(e)}\nonumber\\
\times\prod_{e\in\mathscr{E}(B)}(1-e^{-2a^{15/8}\vec{h}_e})^{\omega_{in}(e)}(e^{-2a^{15/8}\vec{h}_e})^{1-\omega_{in}(e)},\label{eqsd3}
\end{align}
where $\rho^1$ denotes the configuration with every edge open. Suppose $\tilde{\omega}_{in}(e)=\omega_{in}(e)$  for each $e\in \mathscr{B}(D)\cup\mathscr{E}(D)$ except for one edge $e_0$ where $\tilde{\omega}_{in}(e_0)=1$ while $\omega_{in}(e_0)=0$. For any fixed $\omega_{out}$, let $\omega=\omega_{in}\oplus \omega_{out}$ and $\tilde{\omega}=\tilde{\omega}_{in}\oplus\omega_{out}$. If $\omega\in A_{\mathbf{C}}$, then it is not hard to see that
\begin{equation}\label{eqsdweight}
\mathcal{K}\left(B, (\tilde{\omega}\rho^0)_{B}\right)-\mathcal{K}\left(B, (\omega\rho^0)_{B}\right)=\mathcal{K}\left(D, (\tilde{\omega}_{in}\rho^1)_{D}\right)-\mathcal{K}\left(D, (\omega_{in}\rho^1)_{D}\right).
\end{equation}

A key observation is
\begin{equation}\label{eqsdinc}
\{\omega_{out}:\omega_{in}\oplus\omega_{out}\in A_{\mathbf{C}}\}\subseteq \{\omega_{out}:\tilde{\omega}_{in}\oplus\omega_{out}\in A_{\mathbf{C}}\}.
\end{equation}
Combining \eqref{eqsdweight} and \eqref{eqsdinc} with \eqref{eqsd2} and \eqref{eqsd3}, we have that
\[\frac{\mathbb{P}^a_{2/3,f,\vec{h}}(\omega_{in}|A_{\mathbf{C}})}{\mathbb{P}^a_{D,w,\vec{h}}(\omega_{in})}\leq \frac{\mathbb{P}^a_{2/3,f,\vec{h}}(\tilde{\omega}_{in}|A_{\mathbf{C}})}{\mathbb{P}^a_{D,w,\vec{h}}(\tilde{\omega}_{in})},\]
which completes the proof of \eqref{eqsd1} and thus \eqref{eqsd}.
\end{proof}

We are ready to prove Proposition \ref{propdisexp}
\begin{proof}[Proof of Proposition \ref{propdisexp}]
Proposition \ref{propdisexp} follows from Lemma \ref{lemGHS} and Proposition \ref{propcov}.
\end{proof}

\subsection{Proof of Theorem \ref{thm1}}

Proposition \ref{propdisexp} implies: for any $h>0$ and $a\in(0,1]$ we have
\begin{eqnarray}\label{eqlargedis}
\langle \sigma_x;\sigma_y\rangle_{a,h}\leq C_4 a^{1/4} e^{-m_2(h)|x-y|} \text{ whenever }|x-y|>K_2(h) \text{ and }x,y\in a\mathbb{Z}^2
\end{eqnarray}
where $C_4\in(0,\infty)$ is universal, and $m_2(h), K_2(h)\in (0,\infty)$ only depend on h.

For any $x,y\in a\mathbb{Z}^2$ with $|x-y|\leq K_2(h)$, by the GHS inequality \cite{GHS70} and Proposition 5.5 of \cite{DCHN11},
\begin{equation}\label{eqGHS}
\langle \sigma_x;\sigma_y\rangle_{a,h}\leq \langle \sigma_x;\sigma_y\rangle_{a,h=0}\leq \tilde{C}_2a^{1/4}|x-y|^{-1/4},
\end{equation}
where $\tilde{C}_2\in (0,\infty)$. Now, \eqref{eqlargedis} and \eqref{eqGHS} imply
\begin{proposition}\label{propdisexp2}
For any $h>0$ and $a\in(0,1]$ we have
\begin{eqnarray*}
\langle \sigma_x;\sigma_y\rangle_{a,h}\leq C_5(h) a^{1/4}|x-y|^{-1/4} e^{-m_3(h)|x-y|} \text{ for any }x,y\in a\mathbb{Z}^2,
\end{eqnarray*}
where $C_5(h),m_3(h)\in (0,\infty)$ only depend on h.
\end{proposition}
Now we are ready to prove Theorem \ref{thm1}.
\begin{proof}[Proof of Theorem \ref{thm1}]
The leftmost inequalities of \eqref{eqdis1} and \eqref{eqdis2} follow from the Griffiths' inequality (see Corollary 1 of \cite{Gri67}).
So to prove Theorem \ref{thm1}, it remains to show that, for any $a\in(0,1]$ and $h\in(0,a^{-15/8}]$,
\begin{equation}\label{eqthm1proof}
\langle \sigma_x;\sigma_y\rangle_{a,h}\leq C_0a^{1/4}|x-y|^{-1/4}e^{-B_0h^{8/15}|x-y|}\text{ for any }x,y\in a\mathbb{Z}^2.
\end{equation}
In Proposition \ref{propdisexp2}, letting $a=H^{8/15}$ where $H\leq 1$ and $h=1$, we get
\[\langle \sigma_x;\sigma_y\rangle_{H^{8/15},1}\leq C_5(1) H^{2/15}|x-y|^{-1/4} e^{-m_3(1)|x-y|} \text{ for any }x,y\in H^{8/15}\mathbb{Z}^2.\]
Rephrasing the last result on the $\mathbb{Z}^2$ lattice, we get (letting $x^{\prime}=xH^{-8/15}$ and $y^{\prime}=yH^{-8/15}$)
\begin{equation}\label{eqIsingZ2}
\langle \sigma_{x^{\prime}};\sigma_{y^{\prime}}\rangle_{1,H}\leq C_5(1) |x^{\prime}-y^{\prime}|^{-1/4} e^{-m_3(1)H^{8/15}|x^{\prime}-y^{\prime}|} \text{ for any }x^{\prime},y^{\prime}\in \mathbb{Z}^2,
\end{equation}
which proves \eqref{eqdis2}. Now \eqref{eqthm1proof} follows by rephrasing \eqref{eqIsingZ2} on the $a\mathbb{Z}^2$ lattice with external field $ha^{15/8}$.
\end{proof}

\section{Exponential decay and scaling in the continuum}\label{seccon}
\subsection{Exponential decay}
\begin{proof}[Proof of Theorem \ref{thmmain}]
For any $f,g\in C_0^{\infty}(\mathbb{R}^2)$, Theorem 1.4 of \cite{CGN16} plus an extension of Proposition 3.5 of \cite{CGN15} imply
\begin{align}\label{eqmain1}
\lim_{a\downarrow 0}\left[\left\langle\Phi^{a,h}(f)\Phi^{a,h}(g)\right\rangle_{a,h}-\left\langle\Phi^{a,h}(f)\right\rangle_{a,h}\left\langle\Phi^{a,h}(g)\right\rangle_{a,h}\right]
=\text{Cov}\left(\Phi^{h}(f),\Phi^{h}(g)\right).
\end{align}
The extension needed is the replacement in Proposition 3.5 of the magnetization variable $m^a_{\Lambda}=\Phi^{a,h=0}(1_{\Lambda})$ (defined for the measure $\langle \cdot \rangle_{\Lambda}^+$ with plus boundary condition on a square $\Lambda$) by $\Phi^{a,h}(f)$. To verify the extension, choose $\Lambda$ to contain the support of $f$ and note that the GKS inequalities \cite{Gri67a,KS68} imply that
\begin{align*}
\left\langle \exp{(t\Phi^{a,h}(f))} \right\rangle_{a,h}&\leq \left\langle\exp{(\Phi^{a,0}\left((t\|f\|_{\infty}+h)1_{\Lambda}\right))} \right\rangle_{\Lambda}^+=\left\langle e^{\tilde{t}m^a_{\Lambda}} \right\rangle_{\Lambda}^+,
\end{align*}
where $\tilde{t}=t\|f\|_{\infty}+h$.

The LHS of \eqref{eqmain1} before the limit is equal to ($E_h^a(\cdot):=\langle\cdot\rangle_{a,h}$)
\begin{align}\label{eqmain2}
&\Big| E_h^a\Big(a^{15/4}\sum_{x,y\in a\mathbb{Z}^2}\sigma_x f(x) \sigma_y g(y)\Big)-E_h^a\Big(a^{15/8}\sum_{x\in a\mathbb{Z}^2}\sigma_x f(x)\Big)E_h^a\Big(a^{15/8}\sum_{y\in a\mathbb{Z}^2}\sigma_y g(y)\Big) \Big|\nonumber\\
&=\Big| a^{15/4}\sum_{x,y\in a\mathbb{Z}^2}\big[E_h^a\left(\sigma_x f\left(x\right) \sigma_y g\left(y\right)\right)-E_h^a\left(\sigma_x f\left(x\right)\right)E_h^a\left(\sigma_y g\left(y\right)\right)\big]\Big|\nonumber\\
&=\Big| a^{15/4}\sum_{x,y\in a\mathbb{Z}^2}\left[f(x)g(y)\left\langle\sigma_x;\sigma_y\right\rangle_{a,h}\right]\Big|\nonumber\\
&\leq a^{15/4} \sum_{x,y\in a\mathbb{Z}^2} |f(x)g(y)|C_0 a^{1/4}|x-y|^{-1/4}e^{-B_0h^{8/15}|x-y|},
\end{align}
where the last inequality follows from Theorem \ref{thm1} when $0<a\leq \min\{1, h^{-8/15}\}$. Letting $a\downarrow 0$ in \eqref{eqmain2}, and using \eqref{eqmain1} completes the proof.
\end{proof}

\subsection{Scaling of the magnetization fields}\label{secscl}
In \cite{CGN15,CGN16}, the critical and near-critical magnetization fields were denoted by $\Phi^{\infty}$ and $\Phi^{\infty,h}$ (where $h$ is the renormalized magnetic field strength). These are generalized random fields on $\mathbb{R}^2$ so for a suitable test function $f$ on $\mathbb{R}^2$ (including $1_{[-L,L]^2}(x)$), one has random variables $\langle \Phi^{\infty},f\rangle$ (or $\int_{\mathbb{R}^2}\Phi^{\infty}(x)f(x)dx$) and similarly for $\Phi^{\infty,h}$. Here we use $\Phi^0$ and $\Phi^h$ in place of $\Phi^{\infty}$ and $\Phi^{\infty,h}$.
\begin{theorem}\label{thmscl1}
For any $\lambda>0$, the field $\Phi^0_{\lambda}(x)=\Phi^0(\lambda x)$ (i.e., $\langle \Phi^0_{\lambda}, f\rangle =\int_{\mathbb{R}^2}\Phi^0(\lambda x)f(x)dx=\int_{\mathbb{R}^2}\Phi^0(y)f(\lambda^{-1}y)\lambda^{-2}dy=\lambda^{-2}\langle\Phi^0, f_{\lambda^{-1}}\rangle$ with $f_{\lambda^{-1}}(x)=f(\lambda^{-1}x)$) is equal in distribution to $\lambda^{-1/8}\Phi^0(x)$.
\end{theorem}
\begin{proof}
This is a special case of the conformal invariance result (Theorem 1.8) of \cite{CGN15} with the conformal map $\phi(z)=\lambda z$.
\end{proof}

\begin{theorem}\label{thmscl2}
For any $h>0$ and $h_0>0$, the field $\lambda^{1/8}\Phi^{h_0}(\lambda x)$ is equal in distribution to $\Phi^{\lambda^{15/8}h_0}(x)$.
\end{theorem}
\begin{proof}
It follows from \cite{CGN15,CGN16} that the distribution $P_h$ of $\Phi^{h}$ is obtained from $P$ of $\Phi$ by multiplying $P$ by the Radon-Nikodym factor $(1/Z_L)e^{h\langle \Phi, I_{[-L,L]^2}\rangle}$ and letting $L\rightarrow \infty$ --- see, in particular, Section 4 of \cite{CGN16}. Then one applies Theorem \ref{thmscl1} to complete the proof.
\end{proof}

The following observation, which expands on the discussion about scaling relations in the introduction, may be useful to interpret Theorem \ref{thmscl2}.
In the zero-field case, $\Phi^{0}(\lambda x)$ is equal in distribution to $\lambda^{-1/8}\Phi^{0}(x)$ in the sense that, with the change of variables $z=\lambda x$,
\[\int \Phi^{0}(z)f(z)dz = \int \lambda^{-1/8} \Phi^{0}(x) f(\lambda x) \lambda^2 dx = \lambda^{15/8} \int \Phi^{0}(x) f(\lambda x) dx\]
for any $f\in C^{\infty}_0(\mathbb{R}^2)$, where the equalities are in distribution.
In the non-zero-field case, provided that $\tilde{h}=\lambda^{-15/8}h$, using Theorem \ref{thmscl2} one obtains an analogous relation as follows:
\begin{eqnarray*}
\int\Phi^{\tilde{h}}(z)f(z)dz&=&\int\Phi^{\lambda^{-15/8}h}(\lambda x)f(\lambda x) \lambda^2 dx\\
&=&\lambda^{15/8}\int\lambda^{1/8}\Phi^{\lambda^{-15/8}h}(\lambda x)f(\lambda x) dx\\
&=&\lambda^{15/8}\int\Phi^{h}(x)f(\lambda x) dx.
\end{eqnarray*}
Note also that $\tilde{h}=\lambda^{-15/8}h$ implies $M(\Phi^{\tilde{h}}) = C \tilde{h}^{15/8} = \lambda^{-1} M(\Phi^{h})$, where $M$ is introduced in Corollary~\ref{cor:mass}.
This is consistent with the interpretation of $M$ as the inverse of the correlation length.

As noted in Subsection \ref{secoverview}, a version $\Phi^h_{\Omega}$, of $\Phi^h$, can be defined in a (simply connected) domain $\Omega$ (with some boundary condition). In that case, one can consider a conformal map $\phi:\Omega\rightarrow\tilde{\Omega}$ (with inverse $\psi=\phi^{-1}:\tilde{\Omega}\rightarrow\Omega$) and give a generalization of Theorem~\ref{thmscl2}, as we do next. The pushforward by $\phi$ of $\Phi^0_{\Omega}$ to a generalized field on $\tilde{\Omega}$ was described explicitly in Theorem 1.8 of \cite{CGN15}. The generalization to $\Phi^h$, implicit in \cite{CGN16}, is stated explicitly in the next theorem, where we now replace a constant magnetic field $h$ or $\tilde{h}$ on $\Omega$ or $\tilde{\Omega}$ by a suitable magnetic field function $h(z)$ or $\tilde{h}(x)$.
\begin{theorem}\label{thmcon}
The field $\Phi^{h}_{\Omega,\psi}(x):=\Phi^{h}_{\Omega}\left(\psi(x)\right)$ on $\tilde{\Omega}$ is equal in distribution to the field $|\psi^{\prime}(x)|^{-1/8}\Phi^{\tilde{h}}_{\tilde{\Omega}}(x)$ on $\tilde{\Omega}$, where $\tilde{h}(x)=|\psi^{'}(x)|^{15/8}h(\psi(x))$.
\end{theorem}
\begin{proof}
The proof is similar to that of Theorem \ref{thmscl2}, except that one doesn't need to take an infinite volume limit. It is enough to note that, since the pushforward
$\phi*\Phi^{0}_{\Omega}$ is equal in distribution to $|\psi^{'}(x)|^{15/8}\Phi^{0}_{\tilde\Omega}$ (see Theorem 1.8 of \cite{CGN15}), with the choice
$\tilde{h}(x)=|\psi^{'}(x)|^{15/8}h(\psi(x))$, $\tilde{h}(x) \Phi^{0}_{\tilde\Omega}(x)$ is equal in distribution to the pushforward $\phi*(h\Phi^{0}_{\Omega})$.
\end{proof}

\subsection{Proof of Corollary \ref{cor:mass}}\label{seccor2}
\begin{proof}[Proof of Corollary \ref{cor:mass}]
Theorem \ref{thmmain} and the scaling properties of $\Phi^h$ (Theorem \ref{thmscl2}) imply that $M(\Phi^h)=Ch^{8/15}$ where $C>0$. It remains to show that $C<\infty$, or equivalently to rule out the possibility that $M(\Phi^h)$ is infinity. But by \eqref{equpper3.94} in Appendix \upperRomannumeral{2} and the discussion right after \eqref{equpper3.94}, we see that $M(\Phi^h)=\infty$ would imply that
\begin{equation}\label{eqKxy}
\mathscr{K}^h(0,y-x)=\mathscr{K}^h(x,y):=\text{Cov}(\Phi^h(x),\Phi^h(y))=0, \text{ for }y-x\neq0.
\end{equation}
Now $\mathscr{K}^h\geq 0$ and by the GHS inequality \cite{GHS70} together with the fact that $\mathscr{K}^h(x,y)$ is the limit of $\langle \sigma_{x^{\prime}};\sigma_{y^{\prime}}\rangle_{a,h}$ as $a\rightarrow 0$ (with $x^{\prime}\rightarrow x$, $y^{\prime}\rightarrow y$), we see that $\mathscr{K}^h$ is non-increasing in $h\geq 0$. Since, by \cite{Wu66}, $\mathscr{K}^0(x,y)=C^{\prime}|x-y|^{-1/4}$, we would have that for $h\geq0$,
\begin{equation*}
0\leq\mathscr{K}^h(x,y)\leq G_{\epsilon}(y-x):=\begin{cases}
      C^{\prime}|x-y|^{-1/4}, & |y-x|\leq \epsilon,\\
      0, & |y-x|>\epsilon.
   \end{cases}
\end{equation*}
But for $f$ the indicator function $1_{\square}$, of the unit square, we would then have that
\begin{equation}\label{eqVar}
\text{Var}(\Phi^h(1_{\square}))\leq\int\int_{\square\times\square}G_{\epsilon}(y-x) dxdy,
\end{equation}
where the integral is over the product of two unit squares. Since the RHS of \eqref{eqVar} tends to zero as $\epsilon\downarrow0$, we see that $M(\Phi^h)=\infty$ would imply that $\Phi^h(1_{\square})$ is a constant random variable. But this would contradict Proposition 2.2 of \cite{CGN16}.
\end{proof}

\section*{Appendix \upperRomannumeral{1}: Some key ingredients}
\renewcommand*{\thetheorem}{\Alph{theorem}}
\setcounter{theorem}{0}
\renewcommand*{\thecorollary}{\Alph{corollary}}
\setcounter{corollary}{0}
In this appendix, we give exact statements of some key existing results which are major building blocks for the main results of this paper. These include continuum results from \cite{KS19,CCK17} and lattice results from \cite{CDCH16}; precise definitions may be found in these references.

For any bounded $D\subseteq \mathbb{R}^2$, let $D^a:=a\mathbb{Z}^2\cap D$ be its $a$-approximation. Let $L_1, L_2:[0,1]\rightarrow \bar{D}$, the closure of $D$, be two loops. The distance between $L_1$ and $L_2$ is defined by
\[d_{\text{loop}}(L_1,L_2)=\inf\sup_{t\in[0,1]}|L_1(t)-L_2(t)|\]
where the infimum is over all choices of parametrizations of $L_1,L_2$ from the interval $[0,1]$. The distance between two closed sets of loops, $F_1$ and $F_2$, is defined by the Hausdorff metric as follows:
\[d_{\text{LE}}(F_1,F_2)=\inf\{\epsilon>0: \forall L_1\in F_1,~\exists L_2\in F_2 \text{ s.t. } d_{\text{loop}}(L_1,L_2)\leq\epsilon\text{ and vice versa}\}.\]

The following theorem from \cite{KS19} establishes the convergence of the collection of the boundaries of critical FK clusters on the medial lattice (the critical FK loop ensemble, see Section 1.2.2 of \cite{KS19}) to nested CLE$_{16/3}$.
\begin{theorem}[Theorem 1.1 in \cite{KS19}]\label{thmconvcle}
Consider critical FK percolation in a discrete domain $D^a$ with free boundary condition. The collection of the lattice boundaries of critical FK clusters converges in distribution to nested CLE$_{16/3}$ in $D$ in the topology of convergence defined by $d_{\emph{LE}}$.
\end{theorem}

For any configuration $\omega$ in critical FK percolation on $D^a$ with free boundary condition, let $\mathscr{C}(D^a,f,\omega)$ denote the set of clusters of $\omega$ in $D^a$, where $f$ stands for free boundary condition. For $\mathcal{C}\in\mathscr{C}(D^a,f,\cdot)$, let $\mu^a_{\mathcal{C}}:=a^{15/8}\sum_{x\in\mathcal{C}}\delta_x$ be the normalized (by $a^{15/8}$) counting measure of $\mathcal{C}$. For two collections, $\mathscr{S}_1$ and $\mathscr{S}_{2}$, of measures on $D$, the distance between $\mathscr{S}_1$ and $\mathscr{S}_{2}$ is defined by
\begin{equation}\label{eq:metric}
d_{\text{meas}}(\mathscr{S}_1, \mathscr{S}_{2}):=\inf\{\epsilon>0: \forall \mu \in \mathscr{S}_1~\exists \nu \in \mathscr{S}_{2} \text{ s.t. }d_P(\mu, \nu)\leq\epsilon \text{ and vice versa}\},
\end{equation}
where $d_P$ is the Prokhorov distance. The following theorem from \cite{CCK17} establishes convergence of normalized counting measures.
\begin{theorem}[Theorem 8.2 in \cite{CCK17}]\label{thmconvcme}
\[\{\mu^a_{\mathcal{C}}:\mathcal{C}\in\mathscr{C}(D^a,f,\cdot)\}\Longrightarrow \{\mu^0_{\mathcal{C}}:\mathcal{C}\in\mathscr{C}(D,f,\cdot)\},\]
where $\Longrightarrow$ denotes convergence in distribution and the right-hand side is a collection of measures obtained from the scaling limit; here the topology of convergence is defined by  $d_{\emph{meas}}$. Moreover, the joint law of the collection of boundaries of critical FK clusters and $\{\mu^a_{\mathcal{C}}:\mathcal{C}\in\mathscr{C}(D^a,f,\cdot)\}$ converges in distribution to the joint law of CLE$_{16/3}$ and $\{\mu^0_{\mathcal{C}}:\mathcal{C}\in\mathscr{C}(D,f,\cdot)\}$.
\end{theorem}

We also need the following results about the measurability of CME with respect to CLE and the mass of limiting counting measures.
\begin{corollary}[Theorem 8.2 and Lemma 4.16 in \cite{CCK17}]\label{cormeas}
$\{\mu^0_{\mathcal{C}}:\mathcal{C}\in\mathscr{C}(D,f,\cdot)\}$ is measurable with respect to CLE$_{16/3}$ in $D$.
\end{corollary}
\begin{corollary}[Remark 8.3 in \cite{CCK17}]\label{corpostmass}
The mass for each $\mu^0_{\mathcal{C}}$ where $\mathcal{C}$ has positive diameter is strictly positive.
\end{corollary}

The next theorem concerns 6-arm events of type $(100100)$ --- see page 4 of \cite{CDCH16} for the precise definition.
\begin{theorem}\label{thm6arm}
The critical exponent for a 6-arm event of type $(100100)$ is strictly larger than $2$.
\end{theorem}
\begin{proof}
We take $a=1$ in the proof. Let $A_{100100}(0,N)$ be the event that there are 6 disjoint arms $\gamma_k$ from $(0,0)$ or $(\pm1/2,\pm1/2)$ to the boundary of $[-N,N]^2$ which are of type $100100$. Let $I=\{I_k:1\leq k\leq 6\}$ be a family of disjoint arcs on the boundary of $[-1,1]^2$ and $A^I_{100100}(0,N)$ be the event that $A_{100100}(0,N)$ occurs and the arms $\gamma_k$, $1\leq k\leq 6$, can be chosen in such a way that each $\gamma_k$ ends on $NI_k$. To prove the theorem, by quasi-multiplicativity (Theorem 1.3 in \cite{CDCH16}) and Corollary 1.4 of \cite{CDCH16}, it is enough to show that for some $\alpha>0$,
\[\mathbb{P}^1(A^I_{100100}(0,N))\leq C_6 N^{-(2+\alpha)}.\]

 Choose a point $\theta_j$ between $I_j$ and $I_{j+1}$ for $j=1$ and $4$. Conditioned on $A^I_{100100}(0,N)$, the paths $\gamma_1$ and $\gamma_2$ (resp., $\gamma_4$ and $\gamma_5$) can be chosen to be adjacent and jointly form an interface between FK-open and closed regions. With this choice, $\gamma_1, \gamma_2$ (resp., $\gamma_4, \gamma_5$) can be determined by an exploration process starting from $\theta_1$ (resp., $\theta_4$). By conditioning on these two exploration paths and noticing that what happens in the remaining part of $[-N,N]^2$ is FK percolation with inherited boundary conditions, one sees that
\[\mathbb{P}^1(A^I_{100100}(0,N))\leq \mathbb{P}^1(A_0(0,N))\mathbb{P}^1(A^{I\setminus I_6}_{10010}(0,N)).\]
Lemma \ref{lem1arm} now implies that $\mathbb{P}^1(A_0(0,N))\leq C_1 N^{-1/8}$, and Corollary 1.5 of \cite{CDCH16} implies
\[\mathbb{P}^1(A^{I\setminus I_6}_{10010}(0,N))\leq C_7 N^{-2}.\]
Thus the 6-arm critical exponent of type $(100100)$ is at least $17/8$.
\end{proof}

\section*{Appendix \upperRomannumeral{2}: upper bound for the mass}\label{secupper}
In this appendix we give a proof of Theorem \ref{thmupper}. The techniques here are quite different than the FK-based technology used for the proof of Theorem \ref{thm1}. As mentioned in the introduction, an FK-based approach is given in \cite{CJN19}.

Points $x$ in $\mathbb{Z}^2$ will be denoted $x=(k,w)$ with $k,w\in\mathbb{Z}$.
\begin{proof}[Proof of Theorem \ref{thmupper}]
Suppose $\tilde{m}>0$ is as in \eqref{equpper1}; then by the results of \cite{LP68}, for any random variables $F$ and $G$ that are finite linear combinations of finite products of $\sigma_{(0,w)}$'s, one has
\begin{equation}\label{equpper3}
\langle F; T^k G\rangle_{1,H}=\text{Cov}(F,T^k G)\leq C_{F,G}\cdot (e^{-\tilde{m}})^k,
\end{equation}
where $T^k$ translates $G$ $k$ units to the right to be a function of the $\sigma_{(k,w)}$'s. Let $\Sigma_j$ (resp., $\Sigma_{\leq j}$ or $\Sigma_{\geq j}$) denote the $\sigma$-field generated by $\{\sigma_{(j,w)}:w\in\mathbb{Z}\}$ (resp., $\{\sigma_{(k,w)}:w\in\mathbb{Z}, k\leq j~(\text{or } k\geq j)\}$). It follows from the spatial Markov property of our nearest-neighbor Ising model on $\mathbb{Z}^2$, that the random process $X_k=(\sigma_{(k,w)}:w\in\mathbb{Z})$ for $k\in\mathbb{Z}$ is a stationary Markow chain. Let $\mathcal{T}$ denote the transition operator (the transfer matrix in statistical physics terminology); then \eqref{equpper3} may be rewritten (using $(\cdot,\cdot)$ to denote the standard inner product in $\mathcal{H}_0:=L^2(\Omega,P^1_{H},\Sigma_0)$ where $\Omega=\{-1,+1\}^{\mathbb{Z}^2}$) as
\begin{equation}\label{equpper3.7}
(F,(\mathcal{T}^k-\mathcal{P}_1)G)=(F,(\mathcal{T}-\mathcal{P}_1)^kG)\leq C_{F,G}\cdot(e^{-\tilde{m}})^k,
\end{equation}
where $\mathcal{P}_1$ is the orthogonal projection on the eigenspace of constant random variables. Note that $\mathcal{P}_1$ is the same as the expectation with respect to $P^1_{H}$.

Now, by reflection positivity for the Ising model (see, e.g., \cite{GJ87} or \cite{Bis09}), it follows that $\mathcal{T}$ and $\mathcal{T}-\mathcal{P}_1$ are positive semidefinite. By \eqref{equpper3.7}, the spectrum of $\mathcal{T}-\mathcal{P}_1$ is contained in some interval $[0,\lambda]$ with $\lambda\leq e^{-\tilde{m}}$. It follows that \eqref{equpper3.7} is valid for $F, G$ \emph{any} random variables in $\mathcal{H}_0$ and that one may replace $C_{F,G}$ in \eqref{equpper3.7} by
\[\|(I-\mathcal{P}_1)F\|\cdot\|(I-\mathcal{P}_1)G\|,\]
where $\|\cdot\|$ denotes the norm in $\mathcal{H}_0$, so that
\begin{equation}\label{equpper3.8}
\|(I-\mathcal{P}_1)F\|^2=(F,F)-(\mathcal{P}_1F,\mathcal{P}_1F)=E(F^2)-[E(F)]^2=\text{Var}(F),
\end{equation}
where $E$ denotes expectation with respect to $P^1_H$.

If $G\in L^2(\Omega,P^1_{H},\Sigma_{\geq k})$ and $F\in L^2(\Omega,P^1_{H},\Sigma_{\leq 0})$, then by the Markov property
\[E(G|\Sigma_{\leq k})=E(G|\Sigma_{k}):=\tilde{G}=T^{k}\tilde{\tilde{G}}\]
for some $\tilde{\tilde{G}}\in\mathcal{H}_0$, while
\[E(F|\Sigma_{\geq 0})=E(F|\Sigma_{0}):=\tilde{F}\in \mathcal{H}_0.\]
Thus
\begin{equation}\label{equpper3.9}
\text{Cov}(F,G)=\text{Cov}(\tilde{F},\tilde{G})=(\tilde{F},(\mathcal{T}^k-\mathcal{P}_1)\tilde{\tilde{G}})\leq \sqrt{\text{Var}(\tilde{F})}\sqrt{\text{Var}(\tilde{\tilde{G}})} e^{-\tilde{m}k}.
\end{equation}
Since $E(\tilde{F})=E(F)$ while $E(\tilde{F}^2)\leq E(F^2)$ and similarly for $\tilde{\tilde{G}}$,
\begin{equation}\label{equpper3.11}
\text{Cov}(F,G)\leq \sqrt{\text{Var}(F)}\sqrt{\text{Var}(G)}e^{-\tilde{m}k}.
\end{equation}

Using \eqref{equpper3.11} with $F$ and $G$ finite linear combinations of Ising spin variables and recalling \eqref{eqfield}, we have
\begin{equation}\label{equpper3.92}
\text{Cov}(\Phi^{1,H}(\hat{f}),\Phi^{1,H}(\hat{g}))\leq S^1_{H}(\hat{f})S^1_{H}(\hat{g})e^{-\tilde{m}\hat{k}}
\end{equation}
where we write $S^1_{H}(\hat{f})=\sqrt{\text{Var}(\Phi^{1,H}(\hat{f}))}$, provided
\[\text{supp}(\hat{f})\subseteq (-\infty,0]\times\mathbb{R},~\text{supp}(\hat{g})\subseteq [\hat{k},\infty)\times\mathbb{R}.\]
Since $\tilde{M}(H)$ was defined as the supremum of $\tilde{m}$ such that \eqref{equpper1} is valid, \eqref{equpper3.92} is valid with $\tilde{m}$ replaced by $\tilde{M}(H)$.

Suppose that for some $B\in (0,\infty)$ there is a sequence $H_i\downarrow 0$ such that $\tilde{M}(H_i)\geq B H_i^{8/15}$ for all (large) $i$. Then we pick some fixed $h>0$ (say, $h=1$ for convenience), and let $a_i:=(H_i/h)^{8/15}$ so that $a_i\downarrow 0$ and $H_i=h a_i^{15/8}$. Re-expressing \eqref{equpper3.92} in terms of $\Phi^{a_i,h}$ (with $H=H_i$ and $\tilde{m}$ replaced by $\tilde{M}(H_i)$) gives for any $f,g$ whose supports are separated in the $1$-direction by Euclidean distance $\text{sep}(f,g)$, the bound
\[\text{Cov}(\Phi^{a_i,h}(f),\Phi^{a_i,h}(g))\leq S^{a_i}_{h}(f) S^{a_i}_{h}(g)e^{-\tilde{M}(H_i)\text{sep}(f,g)/a_i}.\]
Since $\tilde{M}(H_i)\geq B(h a_i^{15/8})^{8/15}=Bh^{8/15}a_i$, this yields
\begin{equation}\label{equpper3.94}
\text{Cov}(\Phi^{a_i,h}(f),\Phi^{a_i,h}(g))\leq S^{a_i}_{h}(f) S^{a_i}_{h}(g)e^{-Bh^{8/15}\text{sep}(f,g)}.
\end{equation}
In the limit $a_i\downarrow 0$, the mean and second moment (and hence variance and $S^{a_i}_{h}$) have a finite limit (for decent test functions $f$ and $g$ --- see \cite{CGN15}) and so the mass gap $M(\Phi^h)$ in the continuum limit would satisfy $M(\Phi^h)\geq B h^{8/15}$.

If $\limsup_{H\downarrow 0}\tilde{M}(H)/H^{8/15}=\infty$, then one could take $B$ arbitrarily large in \eqref{equpper3.94} which would make $M(\Phi^h)=\infty$. But this would contradict Corollary \ref{cor:mass}. So it follows that $\limsup_{H\downarrow 0}\tilde{M}(H)/H^{8/15}\leq C<\infty$.

\end{proof}

\section*{Acknowledgements}
The research of JJ was partially supported by STCSM grant 17YF1413300 and that of CMN by US-NSF grant DMS-1507019. The authors thank Rob van den Berg, Francesco Caravenna, Gesualdo Delfino, Roberto Fernandez, Alberto Gandolfi, Christophe Garban, Barry McCoy, Tom Spencer, Rongfeng Sun and Nikos Zygouras for useful comments and discussions related to this work. The authors benefitted from the hospitality of several units of NYU during their work on this paper: the Courant Institute and CCPP at NYU-New York, NYU-Abu Dhabi, and NYU-Shanghai. The authors also thank an anonymous referee for a careful reading of the originally submitted version of the paper and suggestions for improvements in presentation, which have been implemented in the current version.

\end{document}